\documentclass[12pt]{amsart}
\usepackage[latin1]{inputenc}
\usepackage{amssymb, amsmath, amscd, mathrsfs,marvosym, psfrag,color,a4} 

\input xy
\xyoption{all}
\DeclareMathAlphabet{\mathpzc}{OT1}{pzc}{m}{it}

\newtheorem{theorem}{Theorem}[section]

\newtheorem{proposition}[theorem]{Proposition}

\newtheorem{lemma}[theorem]{Lemma}

 \newtheorem*{GKM}{The GKM condition}

\theoremstyle{definition}
\newtheorem{definition}[theorem]{Definition}

\theoremstyle{remark}
\newtheorem{remark}[theorem]{Remark}
\newtheorem{remarks}[theorem]{Remarks}

\newcommand{\CA}{{\mathcal A}}

\newcommand{\CG}{{\mathcal G}}

\newcommand{\CI}{{\mathcal I}}
\newcommand{\CJ}{{\mathcal J}}
\newcommand{\CK}{{\mathcal K}}
\newcommand{\CL}{{\mathcal L}}

\newcommand{\CO}{{\mathcal O}}

\newcommand{\CS}{{\mathcal S}}
\newcommand{\CT}{{\mathcal T}}

\newcommand{\CV}{{\mathcal V}}
\newcommand{\CW}{{\mathcal W}}

\newcommand{\CX}{{\mathcal X}}

\newcommand{\CZ}{{\mathcal Z}}

\newcommand{\SM}{{\mathscr M}}
\newcommand{\SN}{{\mathscr N}}

\newcommand{\hCW}{{\widehat\CW}}

\newcommand{\hCS}{{\widehat\CS}}

\newcommand{\tCJ}{{\widetilde{\CJ}}}

\newcommand{\DC}{{\mathbb C}}

\newcommand{\DR}{{\mathbb R}}
\newcommand{\DZ}{{\mathbb Z}}

\newcommand{\DT}{{\mathbb T}}

\newcommand{\bA}{{\mathbf A}}

\newcommand{\bK}{{\mathbf K}}

\newcommand{\bP}{{\mathbf P}}

\newcommand{\bS}{{\mathbf S}}

\newcommand{\Hom}{{\operatorname{Hom}}}

\newcommand{\supp}{{\operatorname{supp}}}

\newcommand{\ol}{\overline}

\newcommand{\id}{{\operatorname{id}}}

\newcommand{\lgl}{\langle}
\newcommand{\rgl}{\rangle}

\newcommand{\comment}[1]{}

\begin{document}

\pagenumbering{arabic}
\title[]{Sheaves on the  alcoves and modular representations I} 

\author[]{Peter Fiebig, Martina Lanini}
\begin{abstract} We consider the set of affine alcoves associated with a root system $R$ as a topological space and define  a certain category $\bS$ of sheaves of $\CZ_k$-modules on this space. Here $\CZ_k$ is the structure algebra of the root system over a field $k$.  To any wall reflection $s$ we associate a wall crossing functor on $\bS$. In the companion article \cite{FieLanModRep} we prove that $\bS$ encodes the simple rational characters of the connected, simply connected algebraic group with root system $R$ over $k$, in the case that $k$ is algebraically closed with characteristic above the Coxeter number. \end{abstract}

\address{Department Mathematik, FAU Erlangen--N\"urnberg, Cauerstra\ss e 11, 91058 Erlangen}
\email{fiebig@math.fau.de}
\address{Universit\`a degli Studi di Roma ``Tor Vergata", Dipartimento di Matematica, Via della Ricerca Scientifica 1, I-00133 Rome, Italy }
\email{lanini@mat.uniroma2.it}
\maketitle

\section{Introduction}
This article is the first in  a series of articles that are meant to introduce and study a new category  that encodes the simple rational characters of a reductive algebraic group in characteristics above the Coxeter number. This category is ``combinatorial'' in the sense that it is  defined in terms of the underlying root system without reference to the group itself. 

The following should give the reader a first idea on the nature of this category. Let $R$ be a root system and denote by $\CA$ the associated set of (affine) alcoves. This set carries a partial order (sometimes called the {\em generic Bruhat order} or the {\em Bruhat order at $-\infty$}) and we obtain a topology on $\CA$ with the order ideals as open sets. Note that the group $\DZ R$ acts on $\CA$ by translating alcoves. Now fix a field $k$ that satisfies the GKM-property with respect to $R$ (i.e., the characteristic is not $2$ and not $3$ if $R$ contains a component of type $G_2$). Denote by $S$ the symmetric algebra over the $k$-vector space associated with the coweight lattice of $R$. The action of the finite Weyl group on the set of $\DZ R$-orbits in $\CA$ gives rise to a commutative $S$-algebra $\CZ=\CZ_k$. 

The category  that we are proposing is a full subcategory  of the category of sheaves\footnote{(in the most ordinary sense)} of $\CZ$-modules on the topological space $\CA$. We denote it by $\bS$. Apart from some minor technical assumptions there are two main properties that we stipulate on objects in $\bS$. The first is the following. Let $x$ be a $\DZ R$-orbit in $\CA$ and let $\SM$ be a presheaf of $\CZ$-modules on $\CA$.  There are two constructions, associated with $x$, that we can perform on $\SM$. One is algebraic, the other topological. By the definition of the structure algebra, the $\DZ R$-orbit $x$ in $\CA$ gives rise to a $\CZ$-module homomorphism $\CZ\to\CZ^x$, where $\CZ^x$ is free as an $S$-module of (graded) rank $1$. We obtain the presheaf $\SM^x$ of $\CZ$-modules on $\CA$ by composing $\SM$ with the functor $\CZ^x\otimes_{\CZ}\cdot$. On the other hand, as $x$ is a subset of $\CA$, we can consider the inclusion $i_x\colon x\to\CA$ and we obtain a natural morphism $\SM^x\to i_{x\ast}i_x^\ast\SM^x$. We say that $\SM$ {\em satisfies the support condition} if this   is an isomorphism for all $\DZ R$-orbits $x$ in $\CA$. 

The second condition that we want the  objects in $\bS$ to satisfy is that they should ``behave well under base change''. Again, our definition of a base change functor is based on an algebraic and a topological construction. Note that the structure algebra is an algebra over the symmetric algebra $S$ over the coroot lattice. If $T$ is a flat $S$-algebra,  then we can extend scalars and obtain  $\CZ_T:=\CZ\otimes_S T$. On the other hand, note that the generic Bruhat order is generated by relations between an alcove and its mirror image with respect to a reflection at an affine root hyperplane. We can coarsen this relation by considering only reflections at hyperplanes corresponding to coroots that are not invertible in $T$, together with translations by arbitrary positive
 roots. This yields a finer topology on $\CA$. We then define a base change functor $(\cdot)\boxtimes_S T$ that maps a presheaf of $\CZ$-modules to a presheaf of $\CZ_T$-modules on the finer topology. It is constructed in such a way that its images satisfy the support condition mentioned above. 
 Now a sheaf ``behaves well under base change'' if it is still a sheaf, and not only a presheaf, after all flat base changes $S\to T$.

Once the category $\bS$ is defined, we introduce a wall crossing functor $\vartheta_s$ on $\bS$ with respect to each wall reflection $s$. It is  constructed on the level of presheaves of $\CZ$-modules  by a simple universal property, but in general it does not preserve  the category of sheaves. However, we show that it preserves the category $\bS$.

In the companion article \cite{FieLanModRep} we consider the exact structure on $\bS$ that is inherited via its inclusion in the abelian  category of sheaves of $\CZ$-modules on $\CA$. We then show that $\bS$ contains enough projectives, and we prove that the ranks of the stalks of the indecomposable projective objects in $\bS$ encode the characters of the simple rational representations of the reductive algebraic groups with root system $R$ if the characteristic of $k$ is larger than the Coxeter number. For this we construct a functor $\Psi$ into the Andersen--Jantzen--Soergel category $\bK$ defined in \cite{AJS} and we show that the indecomposable projective objects are mapped to the indecomposable ``special'' objects in $\bK$\footnote{We should point out right away that our functor $\Psi$ is not fully faithful, not even when restricted to the projective objects in $\bS$. We will show in a forthcoming article that $\bS$ is a filtered category and that  the image of $\Psi\colon\bS\to\bK$ can be thought of as the ``associated graded category''.}.  From this one obtains the irreducible characters of $G$ via known results. 

The advantage of $\bS$ over $\bK$ is  that the definition of the category $\bK$ is very ad-hoc and technical as it is a collection of subgeneric, i.e. local data, without a corresponding global object. Moreover,  the ``special'' objects are defined by applying wall crossing functors to certain base objects and they  do not have an intrinsic categorical characterization such as projectivity.   This makes working with $\bK$ rather difficult, and simple  looking statements often require sophisticated and technical arguments\footnote{An example is the self-duality of special objects studied in  \cite{S}.}. We believe that our category is much easier to work with, as its objects are nothing but sheaves on a topological space, and the objects that are important for representation theoretic applications are defined intrinsically, i.e. without reference to wall crossing functors. 

We have  further hopes with respect to $\bS$. Note that Lusztig's formula for the irreducible characters of $G$ is only valid if the characteristic of $k$ is large enough. Ever since Williamson provided a huge list of examples  for characteristics in which the formula fails (called the {\em torsion primes}), one tries to understand the phenomenon of torsion primes in modular representation theory. As $\bS$ encodes the characters for all characteristics above the Coxeter number, it might be a helpful tool for this. Moreover, we believe that it is possible to simplify the work of Andersen, Jantzen and Soergel by constructing a functor from the category of (deformed) $G_1T$-modules into the associated graded of $\bS$ directly, without refering to $\bK$. This functor then hopefully  makes sense also for restricted critical level representations of the affine Kac-Moody algebra associated with $R$ (so here we should take $k=\DC$), which might lead to a calculation of the irreducible critical characters, i.e. a proof of the Feigin--Frenkel conjecture.


\subsection*{Acknowledgements}
This material is based upon work supported by the National Science Foundation under Grant No. 0932078 000 while the first author was in residence at the Mathematical Sciences Research Institute in Berkeley, California, during the Fall 2014 semester. The second author would like to thank the University of Edinburgh, that supported her research during part of this project, and to acknowledge the MIUR Excellence Department Project awarded to the Department of Mathematics, University of Rome Tor Vergata, CUP E83C18000100006.
 Both authors were partially supported by the DFG grant SP1388.

\section{(Pre-)Sheaves on partially ordered sets}\label{sec-sheavesparord}
This section provides some  basic results on the topology of partially ordered sets and their theory of (pre-)sheaves.  Fix a partially ordered set $(\CX,\preceq)$. For an element $A$ of $\CX$ we will use the short hand notation $\{\preceq A\}=\{B\in\CX\mid B\preceq A\}$. The notations $\{\succeq A\}$, $\{\prec A\}$, etc. have an analogous meaning. 

\subsection{A topology on  $(\CX,\preceq)$}\label{subsec-TopPar}

The following clearly yields a topology on the set $\CX$.
\begin{definition} \label{def-topspac} 
A subset $\CJ$ of $\CX$ is called {\em open}, if $A\in \CJ$ and $B\preceq A$ imply  $B\in \CJ$, i.e. if 
$
\CJ=\bigcup_{A\in\CJ}\{\preceq A\}.
$

\end{definition}
The following statements are easy to check. 
\begin{remark}\label{rem-top}
\begin{enumerate}
\item A subset $\CI$ of $\CX$ is closed if and only if $\CI=\bigcup_{A\in\CI}\{\succeq A\}$.
\item Arbitrary unions and intersections of open sets are open. The same holds for closed sets. 
\item A subset $\CK$ of $\CX$ is locally closed (i.e. an intersection of an open and a closed subset) if and only if $A,B\in\CK$ and $A\preceq C\preceq B$ imply $C\in\CK$. 
\end{enumerate}
\end{remark}
For any subset $ \CT $ of $\CX$ we define 
$$ \CT_{\preceq}:=\bigcup_{A\in \CT }\{\preceq A\}.
$$
This  is the smallest open subset of $\CX$ that contains $ \CT $.


\subsection{Presheaves on partially ordered sets}\label{subsec-PosetSheaves}
Now suppose $\bA$ is an abelian category that has arbitrary products. For a presheaf $\SM$ on $\CX$ with values in $\bA$ and an open subset $\CJ$ of $\CA$ we denote by $\SM{(\CJ)}$ the object of sections of $\SM$ over $\CJ$. For an inclusion $\CJ^\prime\subset\CJ$ we denote by $r_{\CJ}^{\CJ^\prime}\colon \SM(\CJ)\to\SM{(\CJ^\prime)}$ the restriction morphism. Sometimes we write $m|_{\CJ^\prime}$ instead of $r_\CJ^{\CJ^\prime}(m)$. Recall that one calls a presheaf $\SM$ {\em flabby} if for any inclusion $\CJ^\prime\subset\CJ$ of open sets the restriction homomorphism $r_{\CJ}^{\CJ^\prime}$ is surjective.

For a subset $\CT$ of $\CX$ we denote by $i_\CT\colon \CT\to\CX$ the inclusion. 

\begin{definition}\label{def-supppre} Let $\SM$ be a presheaf on $\CX$. \begin{enumerate}\item $\SM$  is said to be {\em supported on $\CT\subset\CX$} if the natural morphism $\SM\to i_{\CT\ast}i_\CT^\ast\SM$ is an isomorphism. \item $\SM$  is said to be  {\em finitary} if $\SM(\emptyset)=0$ and if there exists a finite subset $\CT$ of $\CX$ such that $\SM$ is supported on $\CT$. \end{enumerate}
\end{definition}

(Here, $i_{\CT\ast}$ and $i_\CT^\ast$ are the push-forward and pull-back functors for {\em  presheaves}, i.e., $i_{\CT\ast}\SN(\CJ)=\SN(i_\CT^{-1}(\CJ))$ and $i_\CT^\ast\SM(\CJ)$ is the limit of the system of objects $\SM(\CJ^\prime)$ together with the restriction maps, where $\CJ^\prime$ runs over the set of open neighbourhoods of  $i_{\CT}(\CJ)$). More explicitely, $\SM$ is supported on $\CT$ iff  for any open subset $\CJ$ of $\CX$ the restriction homomorphism $\SM(\CJ)\to\SM((\CJ\cap\CT)_{\preceq})$ is an isomorphism.

The partially ordered sets that we will be interested in in this article will be the sets of alcoves associated with a finite root system endowed with various partial orders depending on the choice of a base ring. 
\section{Alcove Geometry}\label{sec-alcoves}
In the following we review  the basic features of the alcove geometry associated with a root system. As a reference one might consult \cite{Hum}\footnote{In \cite{Hum}, what we refer to as the \emph{affine Weyl group associated with a root system}, would be rather associated with the dual root system. We hope the reader will not be confused by this choice of terminology.}. We  endow the set of alcoves with a topology. This depends on the choice of a {\em base ring} $T$. Then we  study the  decomposition into connected components. The main result is that each connected component is an orbit under a subgroup of the affine Weyl group (this subgroup depends on $T$). We finally introduce {\em admissible} families of open subsets. We prove in a later section that  the sheaves that we are interested in are determined by their restriction to an arbitrary admissible family. 

\subsection{Roots, reflections and alcoves}\label{subsec-Alc}
Fix  a finite irreducible root system $R$ in a real finite dimensional vector space  $V$.  For any $\alpha\in R$  denote by $\alpha^\vee\in V^\ast=\Hom_\DR(V,\DR)$ the corresponding coroot. Let $\langle\cdot,\cdot\rangle\colon V\times V^\ast\to \DR$ be the natural pairing and fix a system $R^+\subset R$ of positive roots. Let
\begin{align*}
X&:=\{\lambda\in V\mid \langle \lambda,\alpha^\vee\rangle\in\DZ\text{ for all $\alpha\in R$}\},\\
X^\vee&:=\{v\in V^\ast\mid \langle\alpha,v\rangle\in\DZ\text{ for all $\alpha\in R$}\}
\end{align*}
be the weight and the coweight lattice, resp. 
For  $\alpha\in R^+$ and $n\in\DZ$  define 
$$
 H_{\alpha,n}:=\{\mu\in V\mid \langle \mu,\alpha^\vee\rangle = n\},
 $$
 the {\em affine reflection hyperplane} associated with $\alpha$ and $n$,  and
\begin{align*}
H_{\alpha,n}^+&:=\{\mu\in V\mid \langle \mu, \alpha^\vee\rangle>n\},\\
H_{\alpha,n}^-&:=\{\mu\in V\mid \langle \mu,\alpha^\vee\rangle<n\},
\end{align*}
the corresponding positive and the negative half-space, resp.
 \begin{definition} 
 The connected components of $V\setminus\bigcup_{\alpha\in R^{+},n\in\DZ}H_{\alpha,n}$ are called {\em alcoves}. We denote by $\CA$ the set of alcoves. 
\end{definition}

Let $\alpha\in R^+$ and $n\in\DZ$.  The {\em affine reflection} with fixed point hyperplane $H_{\alpha,n}$ is 
$$
s_{\alpha,n}\colon V\to V,\quad \lambda\mapsto \lambda-(\langle \lambda,\alpha^\vee\rangle-n)\alpha.
$$
It maps $X$ into $X$ and preserves the set $\{H_{\beta,m}\}$ of affine hyperplanes, hence induces a bijection on the set $\CA$ that we denote by $s_{\alpha,n}$ as well. 

For $\gamma\in X$ we  denote by $t_\gamma\colon V\to V$ the affine translation $\mu\mapsto \mu+\gamma$. Again this preserves the set of hyperplanes and induces  a bijection  $t_\gamma\colon\CA\to\CA$.  
 Easy calculations yield:
\begin{lemma} \label{lemma-easypeasy}
\begin{enumerate}
\item For $\alpha\in R^+$ and $m,n\in\DZ$ we have $s_{\alpha,n}\circ s_{\alpha,m}=t_{(n-m)\alpha}$.
\item For $\alpha\in R^+$,  $n\in\DZ$ and $\lambda\in X$ we have $s_{\alpha,n}\circ t_\lambda= t_{s_{\alpha,0}(\lambda)}\circ s_{\alpha,n}$.
\end{enumerate}
\end{lemma}

Denote by $\hCW$ the affine Weyl group, i.e. the group of affine transformations on $V$ generated by the set $\{s_{\alpha,n}\mid \alpha\in R^+, n\in\DZ\}$.  Lemma \ref{lemma-easypeasy} implies that $t_{\gamma}$ is contained in $\hCW$ for $\gamma\in\DZ R$. The affine Weyl group acts on the set $\CA$, and $\CA$ is a principal homogeneous set for this $\hCW$-action (cf. \cite[Section 4.5]{Hum}).

\subsection{Base rings}\label{subsec-basering} The topology that we introduce on $\CA$ will depend on the choice of what we will call a base ring. First, let $k$ be a field. We denote by $X^\vee_k=X^\vee\otimes_\DZ k$ the $k$-vector space associated with the lattice $X^\vee$. For  $v\in X^\vee$ we often denote by $v$ its canonical image  $v\otimes 1$ in $X^\vee_k$.
We always  assume that the following holds:
\begin{GKM}  The characteristic of $k$ is  $\ne 2$ and for $\alpha,\beta\in R^+$, $\alpha\ne\beta$ we have $\alpha^\vee\not\in k\beta^\vee$ (as elements in $X^\vee_k$).
\end{GKM}

This condition  excludes characteristic $2$ in all cases, and characteristic $3$ in case $G_2$. Note that for the representation theoretic applications we need the additional assumption that the characteristic of $k$ is larger than the Coxeter number of $R$. From now on we assume that $k$ satisfies the above. The main reason for this condition is that it makes sure that in {\em subgeneric situations} (for the definition see below) we deal with $A_1$-type situations, i.e. situations that can be understood very explicitely.

 Let $S=S(X^\vee_k)$ be the symmetric algebra of the $k$-vector space $X^\vee_k$. We consider $S$ as a $\DZ$-graded algebra with degree 2 component $S_2=X^\vee_k\subset S$. Let $T$ be a unital, commutative $S$-algebra that is flat as an $S$-module. Sometimes we assume in addition that $T$ is a $\DZ$-graded $S$-algebra. In this case, a $T$-module is assumed to be graded, and a homomorphism between graded modules is assumed to respect the grading, i.e. it is of degree $0$. For a $\DZ$-graded object $M=\bigoplus_{n\in\DZ} M_n$ and $l\in\DZ$ we write $M[l]$ for the object that we obtain from $M$ by shifting the grading in such a way that $M[l]_n=M_{l+n}$ for all $n\in\DZ$.

Let $T$ be a (not necessarily graded) commutative, unital, flat $S$-algebra. Again we often write $\alpha^\vee$ for the image of $\alpha^\vee$ in $T$.
Note that the flatness implies that the structure homomorphism $S\to T$, $f\mapsto f1_T$, is injective. In particular, for two positive roots $\alpha\ne\beta$, the images in $T$ of $\alpha^\vee$ and $\beta^\vee$ are $k$-linearly independent.

\begin{definition}
\begin{enumerate}
\item Denote by $I_T\subset R^+$ the set of all $\alpha$ such that the left multiplication with $\alpha^\vee$ on $T$ is {\em not} a bijection (i.e. $\alpha^\vee$ is not invertible in $T$).
\item Denote by $\hCW_T$ the subgroup of $\hCW$ that is generated by $t_\gamma$ with $\gamma\in\DZ R$ and by $s_{\alpha,m}$ for all  $\alpha\in I_T$ and $m\in\DZ$. 
\item Denote by $R_T^+\subset R^+$ the set of all $\alpha$ with the property that $\hCW_T$ contains $s_{\alpha,m}$ for some $m\in\DZ$.
\end{enumerate}
\end{definition}
Note that  $\hCW_T$ contains $\DZ R$, hence Lemma \ref{lemma-easypeasy} implies that $s_{\alpha,m}\in\hCW_T$ for some $m\in\DZ$ if and only if this condition holds for all $m\in\DZ$. 
By definition $I_T\subset R_T^+$. But in general this is not a bijection. For example, if $T=S[\alpha^{\vee-1}\mid\text{ $\alpha\in R^+$ is not simple}]$, then $I_T\subset R^+$ is the set of simple roots. It follows that  $\hCW_T=\hCW$, so $R_T^+=R^+$
.
\begin{definition} Call  $T$ {\em saturated} if $I_T=R_T^+$.
\end{definition}

\begin{remark} \begin{enumerate}
\item For $T=S$ we have $I_S=R^+$, $\hCW_T=\hCW$ and $R_S^+=R^+$. Hence $S$ is saturated.
\item Suppose that $I_T=\emptyset$. Then $\hCW_T=\DZ R$ and $R_T^+=\emptyset$. In particular, $T$ is saturated in this case. 
\item Suppose that $I_T=\{\alpha\}$ for some $\alpha\in R^+$. Then $\hCW_T$ is generated by $s_{\alpha,m}$ with $m\in\DZ$ and $t_\gamma$ with $\gamma\in\DZ R$. Then $\hCW_T=\{\id,s_{\alpha,0}\}\ltimes\DZ R$ and $R_T^+=\{\alpha\}$. Again, $T$ is saturated.
\end{enumerate}
\end{remark}

\begin{definition} 
 A {\em base ring} is an $S$-algebra $T$  that satisfies the following:
 \begin{enumerate}
 \item $T$ is unital and commutative.
 \item $T$ is flat as a (graded) $S$-module.
 \item $T$ is saturated.
 \item Any projective $T$-module is free.
 \end{enumerate}
 A base ring $T$ is called {\em generic}, if $R_T^+=\emptyset$, and {\em subgeneric}, if $R_T^+=\{\alpha\}$ for some $\alpha\in R^+$. 
\end{definition}

\subsection{Partial orders and topologies  on $\CA$}
Fix a base ring $T$.
We now endow the set $\CA$ with a partial order that depends on $T$, or rather on the set $I_T=R_T^+$.

\begin{definition}
Denote by  $\preceq_T$  the partial order on $\CA$ that is generated by $A\preceq_T B$ if either $B=t_\gamma(A)$ with $\gamma\in\DZ_{\ge0} R^+$, or $B=s_{\alpha,n}(A)$ with $\alpha\in R_T^+$ and $n\in\DZ$ such that $A\subset H_{\alpha,n}^-$. 
\end{definition}

\begin{remark}\label{rem-ordexp}  If $A\preceq_T B$, then there exists a finite sequence $A=A_0$, $A_1$,\dots,$A_n=B$ such that for all $i=1,\dots,n$ we either have $A_{i+1}=t_\gamma(A_i)$ for some $\gamma\in\DZ_{\ge 0}R^+$, or $A_{i+1}=s_{\alpha,n}(A_i)$ for some $\alpha\in R_T^+$, $n\in\DZ$ with $A_i\subset H_{\alpha,n}^-$. 
\end{remark}

We denote by  $\CA_T$ the topological space associated with $(\CA,\preceq_T)$.  Sometimes we call a subset $\CJ$ of $\CA$ {\em $T$-open} if it is open in $\CA_T$.

\begin{remark}\label{rem-ordorb}
 Let $T\to T^\prime$ be a homomorphism  of base rings. If $\alpha^\vee$ is invertible in $T$, then it is also invertible in $T^\prime$. Hence $I^+_{T^\prime}\subseteq  I^+_T$ and $\hCW_{T^\prime}\subseteq\hCW_T$, so $R_{T^\prime}^+\subseteq R_T^+$. Then  $A\preceq_{T^\prime} B$ implies $A\preceq_T B$, hence a $T$-open subset is also $T^\prime$-open. That means that  the identity  $i_{T^\prime}^T\colon \CA_{T^\prime}\to\CA_T$ is continuous.
\end{remark}

\begin{lemma} \label{lemma-gentop} Let $A\in\CA$ and $\gamma\in\DZ R$. Then $A\preceq_T A+\gamma$ if and only if $\gamma\in\DZ_{\ge 0}R^+$. In particular, the topology induced on a $\DZ R$-orbit in $\CA_T$ via the inclusion is independent of $T$. \end{lemma}
\begin{proof} For $A\in\CA$ denote by $\lambda_A\in V$ the barycenter of $A$. We claim that  $A\preceq_T B$  implies $\lambda_B-\lambda_A\in \DR_{\ge 0} R^+$. In the case $B=t_\gamma(A)$ with $\gamma\in \DZ_{\ge 0}R^+$ we have $\lambda_B-\lambda_A=\gamma\in\DR_{\ge 0}R^+$. If $A=s_{\alpha,n}(B)$ with $A\subset H_{\alpha,n}^-$ we have $\lambda_B=s_{\alpha,n}\lambda_A$ and $s_{\alpha,n}(\lambda_A)-\lambda_A=-(\lgl\lambda_A,\alpha^\vee\rgl-n)\alpha\in\DR_{\ge 0}\alpha$. The general case follows  from Remark \ref{rem-ordexp}. 
Now $\lambda_{A+\gamma}-\lambda_A=\gamma$. Hence $A\preceq_T A+\gamma$  implies that $\gamma\in\DZ R\cap\DR_{\ge 0}R^+=\DZ_{\ge 0}R^+$. The converse is true by definition. 
\end{proof}
In this article we denote $\DZ R$-orbits by lower case Latin letters like $x,y,$ etc.
 
\subsection{Connected components of $\CA_T$} \label{subsec-conncomp}
In the following we denote connected components of $\CA_T$ by upper case Greek letters like $\Lambda$, $\Omega$, etc. 
\begin{lemma}\label{lemma-concomp} The connected components of $\CA_T$ coincide with the $\hCW_T$-orbits in $\CA_T$.
\end{lemma} 
\begin{proof} Let $A,B\in\CA$. Then $A$ and $B$ are contained in the same connected component  if and only if there is a sequence $A=A_0$, $A_1$, \dots, $A_n=B$ in $\CA$ such that for all $i=1,\dots, n$ we have either $A_{i}\preceq_T A_{i+1}$ or $A_{i+1}\preceq_TA_i$.  From Remark \ref{rem-ordexp} it follows that this is the case if and only if there is a sequence $A=A_0$, $A_1$,\dots, $A_n=B$ where $A_{i+1}=t_{\gamma}(A_i)$ for some $\gamma\in\DZ R$, or $A_{i+1}=s_{\alpha,m}(A_i)$ for some $\alpha\in R_T^+$ and $m\in\DZ$. This is the case if and only if $A$ and $B$ are in the same $\hCW_T$-orbit.  
\end{proof}

 Denote by $A_e$ the unique alcove that is contained in the dominant Weyl chamber $\{\lambda\in V\mid \langle \lambda,\alpha^\vee\rgl>0\text{ for all $\alpha\in R^+$}\}$  and contains $0$ in its closure. Then the map 
 $$
 \hCW\to\CA,\quad w\mapsto A_w:=w(A_e)
 $$
 is a bijection as $\CA$ is a principal homogeneous $\hCW$-space (\cite[Section 4.5]{Hum}). 
This allows us to construct the following right action of $\hCW$ on $\CA$. For $A=A_x\in\CA$ and $w\in\hCW$ let $Aw:=A_{xw}$. Clearly, this right action commutes with the left action.
This action is  not continuous, but it preserves the connected components, as we show in part (1) of the following result.  
\begin{lemma}\label{lemma-comp} \begin{enumerate}
\item Let $\Lambda\subset\CA_T$ be a connected component. Then $\Lambda w$ is again a connected component. 
\item Let $T\to T^\prime$ be a homomorphism of base rings. Then every connected component of $\CA_T$ is a union of connected components of $\CA_{T^\prime}$.
\end{enumerate}
\end{lemma}
\begin{proof}
Claim (1) follows from Lemma \ref{lemma-concomp} and the fact that the right action of $\hCW$ on $\CA$ commutes with the left action. Claim  (2) follows from Lemma \ref{lemma-concomp} and the fact that $\hCW_{T^\prime}$ is a subgroup of $\hCW_T$.
\end{proof}

\subsection{The finite Weyl group and $\DZ R$-orbits}\label{subsec-finWeyl}
Denote by $\CV=\CA/\DZ R$ the set of $\DZ R$-orbits in $\CA$ (under the left action), and let $\pi\colon\CA\to\CV$ be the orbit map. We often denote by $\ol A$ or $\ol\Lambda$ the image of $A\in\CA$ or $\Lambda\subset\CA$ in $\CV$. Note that every connected component $\Lambda$ of $\CA_T$ is stable under the action of $\DZ R$ (by Lemma \ref{lemma-concomp}). Hence $\Lambda=\pi^{-1}\pi(\Lambda)$. 

Denote by $\CW\subset\hCW$ the finite Weyl group, i.e the subgroup generated by the $\DR$-linear transformations $s_{\alpha,0}\colon V\to V$ for $\alpha\in R^+$.   Lemma \ref{lemma-easypeasy} implies that the $\hCW$-action on $\CA$ induces a $\hCW$-action on     $\CV$. It has the property that for all $\alpha\in R^+$, $m,n\in\DZ$ and $x\in\CV$ we have $s_{\alpha,n}(x)=s_{\alpha,m}(x)$. As $\hCW=\CW\ltimes\DZ R$ and as $\CA$ is a principal homogeneous $\hCW$-set, $\CV$ is a principal homogeneous $\CW$-set. For convenience, we abbreviate $s_\alpha:=s_{\alpha,0}$ for all $\alpha\in R^+$. 

\begin{lemma} Suppose that $T$ is subgeneric with $R_T^+=\{\alpha\}$. Let $\Lambda$ be a connected component of $\CA_T$. Then $\Lambda$ is the union of  two distinct $\DZ R$-orbits $x$ and $y$ that satisfy $y=s_\alpha x$.
 \end{lemma}
 \begin{proof} $\Lambda$ is an orbit of the subgroup $\hCW_T$ of $\hCW$, hence it is a principal homogeneous $\hCW_T$-set. As $\hCW_T$ is generated by $s_{\alpha,n}$ for $n\in\DZ$ and $\DZ R$ we have $\hCW_T=\{\id,s_{\alpha}\}\ltimes\DZ R$ and the statement follows.
 \end{proof}

\subsection{Admissible families}\label{subsec-admfam} 
We will soon consider sheaves on the topological space $\CA_T$  that satisfy certain conditions. These conditions  assure that they are already determined once their restriction to ``$T$-admissible'' families of open subsets are known. Here is the definition of this notion.  
Suppose that $\DT$ is a family of subsets in $\CA$. 
\begin{definition}\label{def-Tadm} We say that $\DT$ is a  {\em $T$-admissible family}, if it  satisfies the following assumptions.
\begin{enumerate}
\item $\CA\in\DT$.
\item Each element in $\DT$ is $T$-open.
\item $\DT$ is stable under taking finite intersections.
\item For any $T$-open subset $\CJ$ and any $\DZ R$-orbit $x$ in $\CA$  there is some $\tCJ\in\DT$ with $\CJ\cap x=\tCJ\cap x$.
\end{enumerate}
\end{definition}
Here is an example of an admissible family. Let $T\to T^\prime$ be a homomorphism of base rings. 
\begin{lemma}\label{lemma-Top} The family $\DT$ of $T$-open subsets in $\CA$ is $T^\prime$-admissible.
\end{lemma}
\begin{proof} Clearly $\CA$ is $T$-open. Any $T$-open subset is $T^\prime$-open by Remark \ref{rem-ordorb}, and the set of $T$-open subsets is stable under taking (arbitrary) intersections. Property (4) follows from Lemma \ref{lemma-gentop}. \end{proof} 

In Lemma \ref{lemma-rigsinv} we introduce another admissible family.

 \section{The structure algebra}\label{sec-strucalg}
Again we fix a root system $R$. In this section we define the {\em structure algebra} $\CZ_S$ associated with $R$ over the field  $k$. This is a commutative, associative, unital $S$-algebra that is ubiquitous in algebraic Lie theory. For example, it occurs as the torus equivariant cohomology of flag varieties, or as the center of deformed blocks of the category $\CO$ of a semisimple Lie algebra (for $k=\DC$). 
 
 For any base ring $T$ we obtain the $T$-algebra $\CZ_T=\CZ_S\otimes_ST$.  We show that $\CZ_T$ has a canonical decomposition with the direct summands parametrized by the connected components of the topological space $\CA_T$.   Then we study the category of $\CZ_T$-modules. This category will be the target category  for the (pre-)sheaves that we are interested in. We introduce the notion of $\CZ$-support for $\CZ_T$-modules. Later we will define a {\em support condition} for presheaves of $\CZ_T$-modules on $\CA_T$. The main idea of this condition is that the $\CZ$-support and the sheaf theoretic support should be compatible. 
 
\subsection{The structure algebra} 
Recall that we denote by $\CV$  the principal homogeneous $\CW$-set of $\DZ R$-orbits in $\CA$. 
Define 
$$
\CZ_S:=\left\{(z_x)\in\bigoplus_{x\in \CV}S\left|\begin{matrix}\, z_x\equiv z_{s_{\alpha} x}\mod\alpha^\vee \\ \text{ for all $x\in\CV$ and  $\alpha\in R^+$ }\end{matrix}\right\}\right..
$$
More generally, for a subset $\CL$ of $\CV$ we define 
$$
\CZ_S(\CL):=\left\{(z_x)\in\bigoplus_{x\in \CL}S\left|\begin{matrix}\, z_x\equiv z_{s_{\alpha} x}\mod\alpha^\vee \\ \text{ for all $x\in\CL$ and  $\alpha\in R^+$ } \\ \text{ with $s_{\alpha} x\in\CL$}\end{matrix}\right\}\right.
$$

Denote by $\CZ_S^{\CL}$ the image of $\CZ_S$ under the projection $\bigoplus_{x\in\CV}S\to\bigoplus_{x\in\CL}S$  along the decomposition. Then $\CZ_S^{\CL}\subset\CZ_S(\CL)$, but in general, this is a proper inclusion.  Suppose $\CL_1$, \dots, $\CL_n$ are mutually disjoint subsets of $\CV$ with $\CV=\bigcup_{i=1}^n\CL_i$. Then $\CZ_S\subset\bigoplus_{i=1}^n\CZ_S^{\CL_i}\subset\bigoplus_{i=1}^n\CZ_S(\CL_i)\subset\bigoplus_{x\in\CV}S$. Clearly, $\CZ_S^{\{x\}}=S$ for any $x\in\CV$.

Now fix a base ring $T$ and define $\CZ_T:=\CZ_S\otimes_ST$. For a subset $\CL$ of $\CV$ define  $\CZ_T(\CL):=\CZ_S({\CL})\otimes_ST$ and $\CZ_T^\CL:=\CZ_S^\CL\otimes_S T$.  Since $T$ is  flat as an $S$-module we obtain a canonical inclusion $\CZ_T\subset\bigoplus_{x\in\CV} T$, and $\CZ_T^\CL\subset\CZ_T(\CL)$ is the image of $\CZ_T$ in $\CZ_T(\CL)$.

As before we obtain canonical  inclusions  $\CZ_T\subset\bigoplus_{i=1}^n\CZ_T^{\CL_i}\subset\bigoplus_{i=1}^n\CZ_T(\CL_i)\subset\bigoplus_{x\in\CV}T$ for a decomposition $\CV=\dot\bigcup_{i=1}^n\CL_i$, and $\CZ_T^{\{x\}}=T$ for $x\in\CV$. We apply this now to a particular case.
Denote by $C(\CA_T)$ the set of connected components of $\CA_T$. Recall that we denote by $\ol\Lambda$ the image of $\Lambda$ in $\CV$. As each connected component is stable under the action of $\DZ R$,  $\CV$ is the disjoint union of the $\ol\Lambda$, with $\Lambda$ running over the elements in $C(\CA_T)$.

\begin{lemma}  \label{lemma-decZcc}
We have $\CZ_T=\bigoplus_{\Lambda\in C(\CA_T)}\CZ_T({\ol\Lambda})$ and hence $\CZ_T^{\ol\Lambda}=\CZ_T(\ol\Lambda)$ for any $\Lambda\in C(\CA_T)$.
\end{lemma}

\begin{proof}
Let $\rho$ be the product of all coroots that are invertible in $T$. Let $(z_{\ol\Lambda})\in\bigoplus_{\Lambda\in C(\CA_T)}\CZ_S(\ol\Lambda)$. If $x\in\ol\Lambda$ and $\alpha\in R^+$ are such that $s_{\alpha}(x)\not\in\ol\Lambda$, then $\alpha\not\in R_T^+$, as $\Lambda$ is a $\hCW_T$-orbit. Hence $\alpha^\vee$  must be invertible in $T$. It follows that $(\rho z_{\ol\Lambda})$ defines an element in $\CZ_S$, so $(z_{\ol\Lambda})=(\rho z_{\ol\Lambda})\otimes\rho^{-1}$ is contained in $\CZ_T=\CZ_S\otimes_ST$. 
\end{proof}

We also need the following result. 
\begin{lemma}\label{lemma-subgensur} 
Let $x\in\CV$, $\alpha\in R^+$ and $\CL=\{x,s_\alpha x\}$. Then 
$$
\CZ_T^{\CL}=\CZ_T(\CL)=\{(z_x,z_{s_\alpha x})\in T\oplus T\mid z_x\equiv z_{s_\alpha x}\mod \alpha^\vee\}.
$$
\end{lemma}

\begin{proof} 
Denote the right hand side of the alleged equation by $L$. Clearly, $\CZ_T^{\CL}$ is a subset of $L$. As the characteristic of $k$ is not $2$, $L$ is generated by $(1,1)$ and $(\alpha^\vee,-\alpha^\vee)$.  Clearly $(1,1)$ is contained in  $\CZ_T^\CL$. Consider the bijection $\CW\xrightarrow{\sim}\CV$, $w\mapsto w(x)$. The group $\CW$  acts on $X^\vee$ in such a way that  $s_{\beta}(h)=h-\langle\beta,h\rangle\beta^\vee$ for all $\beta\in R^+$ and $h\in X^\vee$. For $y\in\CV$ define $z_{y}=w(\alpha^\vee)$ for $y=w(x)$ with $w\in\CW$. Then $z_{s_\beta w(x)}-z_{w(x)}=-\lgl\beta,w(v)\rgl\beta^\vee\equiv 0\mod\beta^\vee$ for all $w\in\CW$ and $\beta\in R^+$, hence $z=(z_y)\in\CZ_S$.  We have  $z_x=\alpha^\vee$ and  $z_{s_\alpha x}=-\alpha^\vee$. Hence also the second generator of $L$  is contained in $\CZ_T^{\CL}$.
\end{proof}

\subsection{The connection to moment graphs}
Let $\Lambda$ be a connected component of $\CA_T$. We denote by $\CG_\Lambda$ the following moment graph over the lattice $X^\vee$. Its set of vertices is $\ol\Lambda$ (i.e., the set of $\DZ R$-orbits in $\Lambda$), and $x,y\in\ol\Lambda$ are connected by an edge if and only if $y=s_{\alpha}x$ for some $\alpha\in R^+$. This edge is then labeled by $\alpha^\vee$. Note that $x,s_\alpha x\in\ol\Lambda$ imply that $\alpha\in R_T^+$. Then $\CZ_S(\ol\Lambda)$ is, by definition,  the structure algebra over the field $k$ of the moment graph $\CG_\Lambda$. 

The category of $\CZ_S(\ol\Lambda)$-modules is intimately connected to the theory of sheaves on the moment graph $\CG_\Lambda$ (cf. \cite{FieAdv}). In this article we do not refer to sheaves on  moment graphs in order to avoid unnecessary confusion. In a forthcoming work we relate the sheaf category $\bS$ that we are about to define to the  category of moment graph sheaves not on $\CG_\Lambda$, but on an  affine version that is called the {\em stable moment graph} (cf. \cite{LanTAMS}).

\subsection{$\CZ$-modules} \label{subsec-Zmod} Again we fix a base ring $T$. We simplify notation and write $\CZ$, $\CZ^{\CL}$, $\CZ(\CL)$ instead of $\CZ_T$, $\CZ_T^{\CL}$, $\CZ_T(\CL)$.
Let $M$ be a $\CZ$-module. 

\begin{definition} \label{def-rtf} We say that $M$ is {\em root torsion free} if the left multiplication with $\alpha^\vee$ on $M$ is injective for all $\alpha\in R^+$.
\end{definition} 

Denote by $T^{\emptyset}$ the localization of $T$ at the multiplicative set generated by $\{\alpha^{\vee}\}_{\alpha\in R^+}$. Then the canonical homomorphism  $T\to T^{\emptyset}$ is injective,  $I_{T^\emptyset}=\emptyset$, and  $T^{\emptyset}$ is a generic base ring. So each connected component of $\CA_{T^\emptyset}$ is a $\DZ R$-orbit. In particular, Lemma \ref{lemma-decZcc} implies $\CZ_{T^\emptyset}=\bigoplus_{x\in\CV}T^{\emptyset}$. 

Let $M$ be a root torsion free $\CZ$-module.  Then the canonical map $M\to M\otimes_TT^{\emptyset}$, $m\mapsto m\otimes 1$, is injective. As $M\otimes_TT^{\emptyset}$ is a $\CZ_{T^\emptyset}$-module, we obtain a canonical decomposition $M\otimes_TT^{\emptyset}=\bigoplus_{x\in\CV}(M\otimes_TT^{\emptyset})^x$ such that $(z_x)\in \CZ_{T^\emptyset}$ acts on $(M\otimes_TT^{\emptyset})^x$ as multiplication with $z_x$. Let $\CL$ be a subset of $\CV$. We define $M^{\CL}$ as the image of the composition 
$$
M\to M\otimes_TT^{\emptyset}=\bigoplus_{x\in\CV}(M\otimes_TT^\emptyset)^x\to \bigoplus_{x\in\CL}(M\otimes_TT^\emptyset)^x,
$$
where the last map on the right is the projection along the decomposition and the map on the left is  $m\mapsto m\otimes 1$. We write $M^x$ instead of $M^{\{x\}}$ in case of a singleton $\CL=\{x\}\subset\CV$.

\begin{lemma} \begin{enumerate}
\item The $\CZ$-action on $M^{\CL}$ factors over the algebra homomorphism $\CZ\to\CZ^{\CL}$ and induces an isomorphism $\CZ^{\CL}\otimes_\CZ M\xrightarrow{\sim}M^{\CL}$.
\item We have a canonical identification $M=\bigoplus_{\Lambda\in C(\CA_T)}M^{\ol\Lambda}$.
\end{enumerate}
\end{lemma}
\begin{proof} The kernel of $\CZ\to\CZ^{\CL}$ consists of all $z=(z_x)_{x\in\CV}$ with $z_x=0$ for all $x\in\CL$. So the first claim in (1) follows from the fact that $z$ acts on $(M\otimes_TT^{\emptyset})^x$ by multiplication with $z_x$. The identification $M\xrightarrow{\sim}\CZ\otimes_\CZ M$ induces a surjective homomorphism $M^{\CL}\to\CZ^{\CL}\otimes_{\CZ}M$. From the fact that the composition $M^{\CL}\to \CZ^{\CL}\otimes_\CZ M\to M^{\CL}$ (the last map being $z\otimes m\mapsto zm$) is the identity, we deduce the second claim in (1). Claim (2) follows from Lemma \ref{lemma-decZcc}.
\end{proof}

\begin{lemma}\label{lemma-stalksZmod}
Suppose that $M$ is a root torsion free $\CZ$-module.
 \begin{enumerate} \item For any $\CL\subset\CV$, the $\CZ$-module $M^{\CL}$ is root torsion free again. \item If $\{\CL_i\}_{i\in I}$ is a finite set of subsets of $\CV$ with $\CV=\bigcup_{i\in I}\CL_i$, then there is a canonical injective homomorphism $M\to\bigoplus_{i\in I}M^{\CL_i}$.\end{enumerate}
\end{lemma} 
\begin{proof} (1) By definition,  $M^{\CL}$ is a subset of $\bigoplus_{x\in\CL}(M\otimes_TT^{\emptyset})^x$. The module on the right is, as a $T^{\emptyset}$-module, root torsion free, hence so is $M^{\CL}$. 

(2) If $m$ is contained in the kernel of $p^{\CL_i}$ for any $i$, then the arguments above show that $m$ must be in the kernel of the inclusion $M\subset\bigoplus_{x\in\CV}(M\otimes_TT^{\emptyset})^x$, hence $m=0$. 
  \end{proof}
We still assume that $M$ is root torsion free. From the above we obtain a canonical injective homomorphim $M\to \bigoplus_{x\in\CV}M^x$.  This allows us to write each element $m$ of $M$ as a $\CV$-tuple $m=(m_x)_{x\in\CV}$ with $m_x\in M^x$. 

\begin{definition}\begin{enumerate}\item  We say that $M$ is {\em $\CZ$-supported on $\CL\subset\CV$} and we write $\supp_\CZ M\subset\CL$  if the canonical homomorphism $M\to M^{\CL}$ is an isomorphism.\item Let $m\in M$. We say that $m$ is {\em $\CZ$-supported on $\CL$} and we write $\supp_\CZ m\subset\CL$  if $m_x\ne 0$ implies $x\in\CL$. \end{enumerate} \end{definition}
Clearly, $M$ is $\CZ$-supported on $\CL$ if and only if all its elements are $\CZ$-supported on $\CL$.

\begin{lemma}\label{lemma-Zsupp} Let $f\colon M\to N$ be a homomorphism of root torsion free $\CZ$-modules and let $\CL$ be a subset of $\CV$. Then the following are equivalent.
\begin{enumerate}
\item $\ker f$ is $\CZ$-supported on $\CL\subset\CV$.
\item For all   $w\in\CV\setminus\CL$ the induced homomorphism $f^w\colon M^w\to N^w$ is injective.
\end{enumerate}
\end{lemma}
\begin{proof} Let us show that (1) implies (2). So let $w\in\CV\setminus\CL$ and let $m\in M$ be such that $m_w\ne 0$, but  $f^w(m_w)=0$. Let $\delta_w=(\delta_{w,x})\in\bigoplus_{x\in\CV}T$ be defined by $\delta_{w,w}=\prod_{\alpha\in R^+}\alpha^\vee$ and $\delta_{w,x}=0$ for all $x\ne 0$. Then $\delta_{w}\in\CZ$. As $M^w$ is root torsion free, $(\delta_wm)_w=\prod_{\alpha\in R^+}\alpha^\vee m_w\ne 0$. But  $\delta_wm$ is supported on $\{w\}$, so it must be in the kernel of $f$ as $f^w(\prod_{\alpha\in R^+}\alpha^\vee m_w)=0$, which contradicts the assumption in (1). Conversely, suppose that (2) holds. Let $m\in M$ be in the kernel of $f$. We have $f^x(m_x)=0$ for all $x\in\CV$, hence $m_x\ne 0$ implies $x\in\CL$. So $m$ is supported on $\CL$.  \end{proof}

\subsection{Root reflexive $\CZ$-modules}
For $\alpha\in R^+$ define 
$$T^{\alpha}:=T[\beta^{\vee-1}\mid\beta\in R^+,\beta\ne\alpha] 
$$
and, as defined before, 
$$
T^\emptyset:=T[\beta^{\vee-1}\mid \beta\in R^+].
$$  
Then $R^+_{T^\emptyset}=\emptyset$ and  $R^+_{T^\alpha}=\{\alpha\}$ or $R_{T^\alpha}^+=\emptyset$. So a connected component of $\CA_{T^{\emptyset}}$ is a $\DZ R$-orbit, and a connected component of $\CA_{T^\alpha}$ is either a $\DZ R$-orbit, or a union of two $\DZ R$-orbits $x$ and $y$ with $y=s_\alpha x$. If $M$ is  a root torsion free $\CZ$-module, then we can view $M$ as  a subset in $\bigcap_{\alpha\in R^+}(M\otimes_TT^{\alpha})\subset M\otimes_TT^{\emptyset}$. The following definition will become relevant later. 

\begin{definition}  We say that $M$ is {\em root reflexive} if it is root torsion free and $M=\bigcap_{\alpha\in R^+}M\otimes_TT^{\alpha}$.
 \end{definition}

\section{Sheaves on the alcoves} \label{sec-sheavesA}   

In this section we study presheaves on the topological space $\CA_T$ with values in the category of $\CZ_T$-modules.  The sheaves that are relevant for us are finitary and flabby, and their local sections are root torsion free or even root reflexive. But we need two more conditions. The first connects the sheaf theoretic notion of support with the $\CZ$-support defined earlier and is called the {\em support condition}. Once we defined this, we introduce a functor $(\cdot)^+$  that associates with a presheaf $\SM$ its maximal quotient satisfying the support condition. We then prove a ``rigidity'' result: A presheaf satisfying the support condition is already determined by its  restriction to an arbitrary admissible family of open subsets in $\CA_T$.  The category $\bP_T$ contains presheaves of $\CZ$-modules on $\CA_T$ that satisfy the support condition (and some minor technical assumptions). 

In order to state the second condition we construct a base change functor $(\cdot)\boxtimes_TT^{\prime}$ for any flat homomorphism $T\to T^\prime$ of base rings. It incorporates both the extension of scalars functor $(\cdot)\otimes_TT^{\prime}$ and the topological pull-back along $i\colon \CA_{T^\prime}\to\CA_T$, but also the  functor $(\cdot)^+$. For us, the most relevant  category is the category $\bS_T$ of sheaves of $\CZ_T$-modules on $\CA_T$ that satisfy the support condition and have the property that they remain sheaves after base change along  flat homomorphisms of base rings (again, we add some minor technical assumptions).


\subsection{The support condition} We fix a base ring $T$. We say that a presheaf $\SM$ of $\CZ$-modules on $\CA_T$ is {\em root torsion free} if every $\CZ$-module of local sections is root torsion free. Analogously we define \emph{root reflexive} presheaves.

 Let $\SM$ be a  presheaf of $\CZ$-modules on $\CA_T$ that is root torsion free and let $\CL$ be a subset of $\CV$. We  define a new presheaf $\SM^\CL$ by composing with  the functor $(\cdot)^\CL$, i.e. for an open subset $\CJ$ of $\CA_T$ we set
$$\SM^\CL(\CJ):=\SM(\CJ)^\CL=\CZ^{\CL}\otimes_{\CZ}\SM(\CJ)$$ with restriction homomorphism $\id\otimes r_{\CJ}^{\CJ^\prime}$. Again we write $\SM^x$ instead of $\SM^{\{x\}}$ in the case $\CL=\{x\}\subset\CV$. 
Note that the natural homomorphisms $\SM(\CJ)\to\SM(\CJ)^\CL$ combine and yield a morphism $p^{\CL}\colon\SM\to\SM^\CL$ of presheaves.  Lemma \ref{lemma-stalksZmod} implies that there is an injective morphism $\SM\to\bigoplus_{x\in\CV}\SM^x$ of presheaves on $\CA_T$. 

Let $x$ be a $\DZ R$-orbit in $\CA$ and  denote by $i_x\colon x\to\CA_T$ the inclusion. Consider the presheaf $i_{x\ast}i_x^\ast\SM$. Recall that for an open set $\CJ$ we denote by $(\CJ\cap x)_{\preceq_T}$ the smallest $T$-open subset that contains $\CJ\cap x$. Then $$(i_{x\ast}i_x^\ast\SM)(\CJ)=\SM((\CJ\cap x)_{\preceq_T}),$$ and the restriction homomorphism $(i_{x\ast}i_x^\ast\SM)(\CJ)\to (i_{x\ast}i_x^\ast\SM)(\CJ^\prime)$ is the restriction homomorphism $\SM((\CJ\cap x)_{\preceq_T})\to\SM((\CJ^\prime\cap x)_{\preceq_T})$.  The restriction homomorphisms $\SM(\CJ)\to\SM((\CJ\cap x)_{\preceq_T})$ combine and yield a morphism $\SM\to i_{x\ast}i_x^\ast\SM$ of presheaves. Clearly, the functors $(\cdot)^x$ and $i_{x\ast}i_x^\ast(\cdot)$ naturally commute.
The following definition is central for our approach. 

\begin{definition} 
We say that  $\SM$ {\em satisfies the support condition} if $\SM$ is root torsion free and if for any $\DZ R$-orbit $x$ in $\CA$ the natural morphism $\SM^x\to i_{x\ast}i_x^\ast\SM^x$ is an isomorphism. 
\end{definition} 
Explicitely, $\SM$ satisfies the support condition if for any open subset $\CJ$ and any $x\in\CV$ the natural homomorphism $\SM(\CJ)^x\to\SM((\CJ\cap x)_{\preceq_T})^x$ is an isomorphism. 
Here is an equivalent notion of the support condition. 
\begin{lemma} \label{lemma-suppcond} Let $\SM$ be a flabby, root torsion free presheaf of $\CZ$-modules on $\CA_T$. Then the following are equivalent.\begin{enumerate}\item $\SM$ satisfies the support condition.\item  For any inclusion $\CJ^\prime\subset\CJ$ of open sets in $\CA_T$ the kernel of the restriction homomorphism $\SM(\CJ)\to\SM{(\CJ^\prime)}$ is $\CZ$-supported on $\pi(\CJ\setminus\CJ^\prime)$. \end{enumerate}\end{lemma} 
\begin{proof} Suppose  (1)  and let $\CJ^\prime\subset\CJ$ be open. Suppose that $x\in\CV\setminus\pi(\CJ\setminus\CJ^\prime)$. We can identify the homomorphism $\SM(\CJ)^x\to\SM(\CJ^\prime)^x$ with $\SM((\CJ\cap x)_{\preceq_T})^x\to \SM((\CJ^\prime\cap x)_{\preceq_T})^x$. Now $x\not\in\pi(\CJ\setminus\CJ^\prime)$ is equivalent to 
  $\CJ\cap x=\CJ^\prime\cap x$, hence the latter  homomorphism is an isomorphism. So $\SM(\CJ)^x\to\SM(\CJ^\prime)^x$ is an isomorphism for all $x\in\CV\setminus\pi(\CJ\setminus\CJ^\prime)$. By Lemma \ref{lemma-Zsupp}  the kernel of $\SM(\CJ)\to\SM(\CJ^\prime)$ is  $\CZ$-supported on $\pi(\CJ\setminus\CJ^\prime)$. Hence (2) holds.  

Conversely, suppose that  (2) holds for $\SM$. Let $\CJ$ be open and $x\in\CV$. We need to show that $\SM(\CJ)^x\to\SM((\CJ\cap x)_{\preceq_T})^x$, i.e. the $x$-stalk of the restriction homomorphism,  is an isomorphism. As $\SM$ is flabby and the functor $(\cdot)^x$ is right exact, this homomorphism is surjective. Condition (2) implies, via Lemma \ref{lemma-Zsupp}, that it is injective, since $x\not\in\pi(\CJ\setminus(\CJ\cap x)_{\preceq_T})$.  \end{proof}

\begin{definition}
Denote by $\bP=\bP_T$ the full subcategory of the category of presheaves of $\CZ$-modules on $\CA_T$ that contains all objects $\SM$ that are flabby, finitary,  root torsion free and satisfy the support condition.
\end{definition}

 \subsection{The canonical decomposition of objects in $\bP$}
 
 Let $\SM$ be an object in $\bP$. The canonical decomposition of $\CZ$-modules studied in Section \ref{subsec-Zmod} implies that $\SM$ splits into a direct sum $\SM=\bigoplus_{\Lambda\in C(\CA_T)}\SM^{\ol\Lambda}$, where each $\SM^{\ol\Lambda}$ is a presheaf of $\CZ^{\ol\Lambda}$-modules. In general, a presheaf  on a topological space does not split into direct summands  supported on connected components. The support condition in our situation, however, ensures that objects in $\bP$ do. 
 
  \begin{lemma}\label{lemma-supp} Let $\SM$ be an object in $\bP$. For any $\Lambda\in C(\CA_T)$, $\SM^{\ol\Lambda}$ is supported, as a presheaf, on $\Lambda\subset\CA_T$.\end{lemma}
\begin{proof}We need to show that for any open subset $\CJ$ of $\CA_T$, the restriction homomorphism $\SM^{\ol\Lambda}(\CJ)\to\SM^{\ol\Lambda}(\CJ\cap\Lambda)$ is an isomorphism. Since  $\pi(\CJ\setminus(\CJ\cap\Lambda))\subset\CV\setminus{\ol\Lambda}$ Lemma \ref{lemma-suppcond} implies that the kernel of the above restriction is $\CZ$-supported on $\CV\setminus{\ol\Lambda}$. As $\SM^{\ol\Lambda}(\CJ)$ is, by definition, $\CZ$-supported on ${\ol\Lambda}$, the kernel of the restriction is trivial. The restriction is surjective as $\SM$ is flabby. \end{proof}

\subsection{The maximal quotient satisfying the support condition}

Let $\SM$ be a flabby and  root torsion free presheaf of $\CZ$-modules on $\CA_T$. For any $x\in\CV$ and any open subset $\CJ$ consider the composition  $\SM(\CJ)\to\SM((\CJ\cap x)_{\preceq_T})\to\SM((\CJ\cap x)_{\preceq_T})^x=i_{x\ast}i_x^\ast\SM^x(\CJ)$, where the first homomorphism is the restriction and the second the natural homomorphism onto the $x$-stalk. These homomorphisms combine and yield a morphism $\gamma^x\colon\SM\to i_{x\ast}i_x^\ast\SM^x$. Denote by $\gamma\colon\SM\to\bigoplus_{x\in\CV}i_{x\ast}i_x^\ast\SM^x$ the direct sum. Define $\SM^+$ as the image (in the category of presheaves) of $\gamma$, i.e. $\SM^+(\CJ)$ is the image of $\SM(\CJ)$ in $\bigoplus_{x\in\CV}\SM((\CJ\cap x)_{\preceq_T})^x$. This is a flabby and root torsion free presheaf again (cf. Lemma \ref{lemma-stalksZmod}). We now show that it also satisfies the support condition. 

 \begin{lemma} \label{lemma-inj} Let $\SM$ be a flabby, root torsion free presheaf.
 \begin{enumerate}
 \item The presheaf $\SM^+$ satisfies the support condition.
 \item $\SM$ satisfies the support condition if and only if $\gamma$ induces an isomorphism $\SM\cong\SM^+$, i.e. if and only if $\gamma$ is injective (in the category of presheaves).
 \end{enumerate} 
\end{lemma} 

\begin{proof} (1) Note that for any $x\in\CV$ we obtain an injective morphism $\SM^{+x}\to i_{x\ast}i_x^{\ast}\SM^x$ (as $\SM^+\subset\bigoplus_{y\in\CV}i_{y\ast}i^\ast_y\SM^y$ and as each local section of $i_{y\ast}i_y^{\ast}\SM^y$ is $\CZ$-supported on $y$). As $\SM$ is flabby, this morphism is also surjective, hence an isomorphism. This implies  $\SM^{+x}\cong i_{x\ast}i_x^{\ast}\SM^{+x}$. Hence $\SM^+$ satisfies the support condition. 

(2) If $\SM\cong\SM^+$, then $\SM$ satisfies the support condition by (1). Conversely, suppose that $\SM$ satisfies the support condition. The direct sum of the canonical homomorphisms onto the stalks $\SM\to\bigoplus_{x\in\CV}\SM^x$ is injective. Hence the composition $\SM\to\bigoplus_{x\in\CV}\SM^x\xrightarrow{\sim}\bigoplus_{x\in\CV} i_{x\ast}i_x^\ast\SM^x$ is injective. Hence $\SM\cong\SM^+$. 
\end{proof}
Let us denote by $\gamma\colon\SM\to\SM^+$ the induced morphism as well. The following is a universal property for $\gamma$. 
\begin{proposition} Let $f\colon \SM\to\SN$ be a morphism of flabby, root torsion free presheaves and suppose that $\SN$ satisfies the support condition. Then there is a unique morphism $g\colon \SM^{+}\to\SN$ such that $f=g\circ \gamma$. 
\end{proposition} 

\begin{proof}  Note that $f$ induces a morphism $\bigoplus_{x\in\CV} i_{x\ast}i_x^\ast\SM^x\to  \bigoplus_{x\in\CV} i_{x\ast}i_x^\ast\SN^x$ such that the diagram

\centerline{
\xymatrix{
 \SM\ar[d]_{\gamma^\SM}\ar[r]^f&\SN\ar[d]^{\gamma^\SN}\\
 \bigoplus_{x\in\CV} i_{x\ast}i_x^\ast\SM^x\ar[r]&  \bigoplus_{x\in\CV} i_{x\ast}i_x^\ast\SN^x
}
}
\noindent 
commutes. By Lemma \ref{lemma-inj}, the morphism $\gamma^\SN$ is injective. Hence the kernel of $\gamma^\SM$ is contained in the kernel of $f$, so $f$ factors over the image of $\gamma^\SM$. \end{proof}


\subsection{Rigidity}
Recall the definition of a $T$-admissible family in Definition \ref{def-Tadm}. 
We now show that a morphism in $\bP$ is already determined by its restriction to  elements in a $T$-admissible family.
\begin{proposition}\label{prop-rig} Suppose that $\DT$ is a $T$-admissible family. Let  $\SM$ and $\SN$ be objects in $\bP$.  Suppose we are given  a homomorphism $f^{(\CJ)}\colon \SM {(\CJ)}\to \SN {(\CJ)}$ of $\CZ$-modules for each $\CJ\in\DT$ in such a way that for any inclusion $\CJ^\prime\subset\CJ$ of sets in $\DT$   the diagram

\centerline{
\xymatrix{
\SM {(\CJ)}\ar[rr]^{f^{(\CJ)}}\ar[d]_{r_\CJ^{\CJ^\prime}}&&\SN {(\CJ)}\ar[d]^{r_{\CJ}^{\CJ^\prime}}\\
\SM {(\CJ^\prime)}\ar[rr]^{f^{(\CJ^\prime)}}&&\SN {(\CJ^\prime)}
}
}
\noindent 
commutes. Then there is a unique morphism $f\colon \SM \to \SN $ of presheaves such that $f{(\CJ)}=f^{(\CJ)}$ for all $\CJ\in\DT$. 
\end{proposition}

\begin{proof} Note that the uniqueness follows immediately from the flabbiness of $\SM$ together with the fact that $\CA\in\DT$.  It remains to show the existence.  In a first step we construct a morphism $f^x\colon\SM^{ x}\to\SN^{ x}$ for any $x\in\CV$. For $\CJ$ open in $\CA_T$ let $\tCJ\in\DT$ be such that  $\CJ\cap x=\tCJ\cap x$. Then $\SM^{ x}(\CJ)=\SM((\CJ\cap x)_{\preceq_T})^x=\SM((\tCJ\cap x)_{\preceq_T})^x=\SM^x(\tCJ)$ and, similarly,  $\SN^{ x}(\CJ)\cong\SN^x(\tCJ)$. So  there is a unique homomorphism  $f^{ x}(\CJ)\colon \SM^{ x}{(\CJ)}\to\SN^{ x}(\CJ)$ induced by $(f^{(\tCJ)})^x$ using these identifications. If $\tCJ^\prime\subset\tCJ$ is another element in $\DT$ that satisfies $\CJ\cap x=\tCJ^\prime\cap x$, then the restriction homomorphisms induce isomorphisms $\SM(\tCJ)^x\xrightarrow{\sim}\SM(\tCJ^\prime)^x$ and  $\SN(\tCJ)^x\xrightarrow{\sim}\SN(\tCJ^\prime)^x$ by Lemma \ref{lemma-suppcond}, and $(f^{(\tCJ)})^x\cong (f^{(\tCJ^\prime)})^x$ under these identifications. The homomorphism $f^x(\CJ)$ constructed above is hence independent of the choice of $\tCJ$. Clearly we have now constructed  a morphism $f^{ x}\colon\SM^{ x}\to\SN^{ x}$ of presheaves on $\CA_T$.

Again let $\CJ$ be an arbitrary open subset of $\CA_T$. Consider the following commutative diagram:

\centerline{
\xymatrix{
\SM(\CA)\ar[d]^{f^{(\CA)}}\ar[r]^{r_\CA^\CJ}&\SM(\CJ)\ar[r]&\bigoplus_{x\in\CV}\SM^{ x}(\CJ)\ar[d]^{f^{ x}(\CJ)}\\
\SN(\CA)\ar[r]^{r_\CA^\CJ}&\SN(\CJ)\ar[r]&\bigoplus_{x\in\CV}\SN^{ x}(\CJ).\\
}
}
\noindent
Note that the two horizontal maps on the right are injective as $\SM$ and $\SN$ are root torsion free. As $\SM$ is flabby, the upper left horizontal map is surjective. It follows that there is a unique homomorphism $f(\CJ)\colon\SM(\CJ)\to\SN(\CJ)$ completing the diagram. The uniqueness statement also ensures that all  homomorphisms $f(\CJ)$ thus constructed yield a morphism $f\colon \SM\to\SN$ of presheaves. If $\CJ$ is in $\DT$ and if we put $f^{(\CJ)}$ in the middle of the above diagram, then it commutes by the construction of the $f^{ x}$. Hence  $f(\CJ)=f^{(\CJ)}$.  
\end{proof}

\subsection{Base change}
Suppose that $T\to  T^\prime$ is a flat  homomorphism of  base rings. In this section we want to construct a base change functor $\cdot\boxtimes_TT^\prime$ from $\bP_T$ to $\bP_{T^\prime}$. Let $\SM$ be an object in $\bP_T$. As a first step define $\SM\otimes_TT^\prime$ as the presheaf of $\CZ_{T^\prime}$-modules on $\CA_T$ with sections $(\SM\otimes_TT^\prime)(\CJ)=\SM(\CJ)\otimes_TT^\prime$ for any $T$-open subset $\CJ$ and  with the obvious restriction homomorphisms. 
Denote by $i\colon\CA_{T^\prime}\to\CA_T$ the (continuous) identity map.  Again we denote by $i^{\ast}$ the presheaf-theoretic pullback functor. Note that it naturally commutes with base change functors. But  in general the functor $i^\ast(\cdot)\otimes_TT^\prime$ maps an object in $\bP_T$ only to a presheaf of $\CZ_{T^\prime}$-modules on $\CA_{T^\prime}$, not necessarily to an object in $\bP_{T^\prime}$. However, every object in the image clearly is finitary, flabby and root torsion free. Hence the only missing property is the support condition. The following statement defines and characterizes  the base change functor. 

\begin{proposition}\label{prop-defbox} 
There is an up to isomorphism unique pair ($\cdot\boxtimes_TT^\prime$,  $\tau$), where $\cdot\boxtimes_TT^\prime\colon \bP_T\to\bP_{T^\prime}$ is a functor and $\tau\colon i^\ast(\cdot)\otimes_TT^\prime\to(\cdot)\boxtimes_TT^\prime$ is a natural transformation with the following property. For any $T^\prime$-open subset $\CJ$ that is also $T$-open $\tau$ induces an isomorphism   $i^\ast\SM\otimes_TT^\prime(\CJ)\xrightarrow{\sim}\SM\boxtimes_TT^\prime(\CJ)$.
\end{proposition} 

\begin{proof} As the set of $T$-open subsets is $T^\prime$-admissible by Lemma \ref{lemma-Top}, the uniqueness follows from Proposition \ref{prop-rig}. We need to show the existence. For an object $\SM$ in $\bP_T$ set
$$
\SM\boxtimes_TT^\prime:=(i^\ast(\SM\otimes_TT^{\prime}))^{+}.
$$
More explicitely, if $\CJ$  is $T^\prime$-open, then $(\SM\boxtimes_TT^\prime)(\CJ)$ is the image of $\SM(\CJ_{\preceq_T})\otimes_TT^\prime$ in $\bigoplus_{x\in\CV}\SM((\CJ_{\preceq_T}\cap x)_{\preceq_T})^{x}\otimes_TT^\prime$. The natural transformation $\tau\colon i^\ast(\SM)\otimes_TT^\prime=i^{\ast}(\SM\otimes_TT^\prime)\to \SM\boxtimes_TT^\prime$ is  induced by the natural transformation  $\gamma\colon \id\to(\cdot)^+$  defined earlier. We already observed that $\SM\boxtimes_TT^{\prime}$ is flabby, finitary and root torsion free. As it satisfies the support condition by construction, it is an object in $\bP_{T^\prime}$. It remains to show that the pair $(\cdot\boxtimes_TT^{\prime},\tau)$  has the required property.

Suppose that $\CJ$ is $T$-open. Then $i^{\ast}\SM(\CJ)=\SM(\CJ)$. As $\SM$ satisfies the support condition, $\SM(\CJ)\to\bigoplus_{x\in\CV}\SM((\CJ\cap x)_{\preceq_T})^x$ is injective. As $T\to T^\prime$ is flat, the induced homomorphism  $\SM(\CJ)\otimes_TT^\prime\to \bigoplus_{x\in\CV}\SM((\CJ\cap x)_{\preceq_T})^x\otimes_TT^\prime$ is injective. Hence $\SM(\CJ)\otimes_TT^\prime$ is isomorphic to the image of this homomorphism, which, by definition, is $(\SM\boxtimes_TT^\prime)(\CJ)$. (Note that the extension of scalars functor $\otimes_TT^{\prime}$ commutes with $i^\ast$ and with $(\cdot)^+$.)
\end{proof}

\begin{remarks}\label{rem-propbox} \begin{enumerate}
\item For flat homomorphisms $T\to T^\prime$ and $T^\prime\to T^{\prime\prime}$ it follows from the defining property that we have an isomorphism $(\SM\boxtimes_TT^\prime)\boxtimes_{T^\prime}T^{\prime\prime}\cong \SM\boxtimes_TT^{\prime\prime}$.
\item Suppose that $T\to T^\prime$ is a flat homomorphism of base rings with the property that the identity $\CA_{T^\prime}\to\CA_T$ is a homeomorphism of topological spaces. Then  $\SM\boxtimes_TT^\prime\cong\SM\otimes_TT^\prime$.
\end{enumerate}
\end{remarks}

 If $\SM$ is a sheaf on $\CA_T$ and $T\to T^\prime$ is a flat homomorphism of base rings, then $\SM\boxtimes_TT^\prime$ is not necessarily a sheaf on $\CA_{T^\prime}$. This leads us to the following definition. 
\begin{definition}\label{def-C} Denote by $\bS=\bS_T$ the full subcategory of the category $\bP_T$ that contains all objects $\SM$ that satisfy the following properties: \begin{enumerate} \item $\SM$ is a root reflexive sheaf on $\CA_T$. \item For all flat homomorphisms $T\to T^\prime$ of base rings the presheaf $\SM\boxtimes_TT^\prime$ is a root reflexive sheaf on $\CA_{T^\prime}$. \end{enumerate}
\end{definition}

 \section{Reflections on alcove walls}  As a preparation for the wall crossing functors, we return in this section  to basic alcove geometry. We consider reflections at alcove walls (i.e. the right action of simple  reflections in the affine Weyl group).  The first result is about the partial order and it is well known.  We use it to deduce some topological statements. Then we discuss subsets that are invariant under alcove wall reflections on the set level and also invariants in the structure algebra. We finish this section with  twisting functors for certain presheaves
that are associated with wall reflections. 
 
  \subsection{Reflections at alcove walls} Fix a base ring $T$.  
Denote by $\hCS$ the set of reflections in $\hCW$ at hyperplanes that have a codimension $1$ intersection with the closure of the fundamental alcove  $A_e$. The elements in $\hCS$ are called the {\em simple affine reflections}. Fix $s\in\hCS$. Recall the right action $A\mapsto As$ that we defined in Section \ref{subsec-conncomp}. 
 For any connected component $\Lambda$ of $\CA_T$ the set $\Lambda s$ is a connected component again by Lemma \ref{lemma-comp}.  

\begin{lemma}\label{lemma-tops} Let $\Lambda$ be a connected component of $\CA_T$. 
\begin{enumerate}
\item Suppose that  $\Lambda=\Lambda s$ and let $A\in \Lambda$. Then $A,As$ are
$\preceq_{T}$-comparable and $\{A,As\}$ is a
$\preceq_{T}$-interval. Moreover, for all $B\in\Lambda$ the following holds.
\begin{enumerate}
\item If $As\preceq_{T} A$ and $B\preceq_{T}A$ then
$Bs\preceq_{T}A$,
\item If $As\preceq_{T} A$ and $As\preceq_{T}B$ then
$As\preceq_{T}Bs$.
\end{enumerate}
\item Suppose that $\Lambda\ne\Lambda s$. Then the map $\Lambda\to\Lambda s$, $A\mapsto As$, is a homeomorphism of topological spaces.
\end{enumerate}
\end{lemma}
\begin{proof}

(1) We have $As=s_{\alpha,n}(A)$ for some $\alpha\in R^+$ and $n\in\DZ$. As $\CA$ is a principal homogeneous $\hCW$-space and as $A,As$ are contained in the same $\hCW_T$-orbit $\Lambda$, this implies $s_{\alpha,n}\in\hCW_T$, i.e. $\alpha\in R^+_T=I_T$, hence $A$ and $As$ are $\preceq_T$-comparable. Now note  that in the case $T=S$ we have $I_T=R_T^+=R^+$ and $\hCW_T=\hCW$.  The partial order $\preceq_S$ hence coincides with the partial order on $\CA$ considered in \cite{LusAdv}. Denote by $d\colon\CA\times\CA\to\DZ$ the distance function defined in \cite[Section 1.4]{LusAdv} (it is a weighted counting of the hyperplanes dividing two alcoves). Then $d(C,D)=-d(D,C)$, $d(C,Cs)=\pm 1$, $d(B,C)+d(C,D)+d(D,B)=0$ for arbitrary alcoves $B,C,D$. Moreover, $B\prec_S C$ implies $d(B,C)>0$. From these statements it easily follows that $\{A,As\}$ is a $\preceq_S$-interval. The statements (1a) and (1b)  are proven in \cite[Proposition 3.2, Corollary 3.3]{LusAdv}. Moreover, as $T$ is an $S$-algebra, $C\preceq_T D$ implies $C\preceq_S D$. As $\{A,As\}$ is a $\preceq_S$-interval, it must be a $\preceq_T$-interval as well.

(1a)  We can assume $B\ne A$ and $B\preceq_T Bs$. Note that by the above arguments, the corresponding result in \cite{LusAdv} implies  $Bs\preceq_S A$. Since $B\prec_T A$, there is a sequence 
$$
B=B_0\preceq_T B_1\preceq_T\ldots \preceq_T B_{n-1}\preceq_T B_n=A,
$$
where either $B_i=s_{\alpha_i, m_i} B_{i-1}$ for $\alpha_i\in R^+_T$ and $m_i\in\DZ$, or $B_i=B_{i-1}+\gamma$, for $\gamma\in\DZ_{\geq 0}R^+$.  We prove that $Bs\preceq_T A$ by induction on $n$. Suppose $n=1$. If $A=t_{\gamma}(B)$, then $Bs\preceq_T As$ as $t_\gamma$ preserves $\preceq_T$. Hence $Bs\preceq_T A$. If $A=s_{\alpha,m}(B)$, then $As=s_{\alpha,m}(Bs)$. Hence $As$ and $Bs$ are $\preceq_T$-comparable. If $Bs\preceq_T As$, then $Bs\preceq_T A$. If  $As\preceq_T Bs$, then $As\preceq_S Bs$. As $Bs\preceq_S A$ and $B\ne A$ and as $\{A,As\}$ is a $\preceq_S$-interval, we deduce  $B=As$, i.e. $Bs=A\preceq_T A$. 

So suppose $n>1$. If $B_1s\preceq_T B_1$, then our arguments above prove $Bs\preceq_T B_1$, hence $Bs\preceq_T A$. If $B_1\preceq_T B_1s$ the induction hypothesis yields $Bs\preceq_T B_1 s$ and $B_1s\preceq_T A$, hence $Bs\preceq_T A$.

(1b) is proven with similar arguments.

(2) As $\Lambda$ and $\Lambda s$ play symmetric roles,  it suffices to show that for $A,B\in\Lambda$ with $A\preceq_T B$ it follows that $As\preceq_T Bs$. So suppose $A\preceq_T B$. We can assume that $B=A+\gamma$ for some $\gamma\ge 0$ or $A=s_{\alpha,n}(B)$ for some $\alpha\in R^+_T$ with $A\subset H_{\alpha,n}^-$. In the first case we have $Bs=As+\gamma$, hence $As\preceq_T Bs$. In the second case, $As=s_{\alpha,n}(Bs)$, hence $As$ and $Bs$ are comparable. We need to show $As\subset H_{\alpha,n}^-$. If this were not the case, then the straight line between the barycenters of $A$ and $As$ would pass through $H_{\alpha,n}$. As $A$ and $As$ share a wall, this means that this wall would be contained in $H_{\alpha,n}$, i.e. $As=s_{\alpha,n}(A)$. So $A$ and $As$ would be in the same $\hCW_T$-orbit, as $\alpha\in R_T^+$, hence in the same connected component which contradicts the assumption $\Lambda\ne\Lambda s$. 
\end{proof}

\subsection{$s$-invariant subsets}
A subset $\CT$ of $\CA$ is called {\em $s$-invariant} if $\CT=\CT s$. For any subset $\CT$ set $\CT^\sharp=\CT\cup\CT s$ and $\CT^\flat=\CT\cap\CT s$. These are the smallest $s$-invariant subset of $\CA$ containing $\CT$ and the largest $s$-invariant subset of $\CT$, resp. 
 
\begin{lemma}\label{lemma-Js} Suppose that $\CJ$ is open in $\CA_T$. Then $\CJ^\sharp$ and $\CJ^\flat$ are open as well.
\end{lemma}
\begin{proof} We can assume that $\CJ$ is contained in a connected component $\Lambda$ of $\CA_T$. If $\Lambda\ne\Lambda s$ then  Lemma \ref{lemma-tops}, (2) implies that $\CJ$ and $\CJ s$ are open, hence $\CJ^\sharp$ and $\CJ^\flat$ are open as well.  Now assume  $\Lambda=\Lambda s$. Let $A\in\CJ^\sharp$ and suppose that $B\preceq_T A$. If $A\in\CJ$, then $B\in\CJ$ as $\CJ$ is open, and hence $B\in\CJ^\sharp$. So suppose that $A\in\CJ s$. If $Bs\preceq_T B$, then $Bs\preceq_T A$ and hence $Bs\preceq_T As$ by  Lemma \ref{lemma-tops}, (1b) (note that the roles of $A$ and $B$ are reversed in the statement there). So $Bs\in\CJ$. If $B\preceq_T Bs$, then $B\preceq_T As$, again by Lemma \ref{lemma-tops}, (1b). So $B\in\CJ$.  In any case, $B\in\CJ^\sharp$. Hence $\CJ^\sharp$ is open.

If $A\in\CJ^\flat$ and $B\preceq_T A$, then $B\in\CJ$. If $As\preceq_T A$, then $Bs\preceq_T A$ by Lemma \ref{lemma-tops}, (1a). Hence $Bs\in\CJ$, so $B\in\CJ s$ and hence $B\in\CJ^\flat$. If $A\preceq_TAs$, then $B\preceq_T As$ and so $Bs\preceq_T As$, again by Lemma \ref{lemma-tops}, (1a). From $As\in\CJ$ we deduce $Bs\in\CJ$, hence   $B\in\CJ s$. So  $B\in\CJ^\flat$.
\end{proof}

 \begin{lemma} \label{lemma-sinvset} Let $\Lambda$ be a connected component of $\CA_T$ and suppose that $\Lambda=\Lambda s$.
\begin{enumerate}
\item Let $\CJ$ be an open subset of $\Lambda$. Then $A\in\CJ^\sharp\setminus\CJ$ implies $As\preceq_T A$ and $A\in\CJ\setminus\CJ^\flat$ implies $A\preceq_T As$.
\end{enumerate}
Let $\CJ_i$ be a family of open subsets of $\Lambda$. 
\begin{enumerate}\setcounter{enumi}{1}
\item We have   $\bigcup_{i\in I}\CJ_i^\sharp=(\bigcup_{i\in I}\CJ_i)^\sharp$ and $\bigcap_{i\in I}\CJ_i^\sharp=(\bigcap_{i\in I}\CJ_i)^\sharp$. 
\item We have  $\bigcup_{i\in I}\CJ_i^\flat=(\bigcup_{i\in I}\CJ_i)^\flat$ and $\bigcap_{i\in I}\CJ_i^\flat=(\bigcap_{i\in I}\CJ_i)^\flat$.
\end{enumerate}
\end{lemma}
\begin{proof} 
The assumption $\Lambda=\Lambda s$ implies that $A$ and $As$ are comparable for all $A\in\Lambda$ by Lemma \ref{lemma-tops}. 

(1) Let $A\in\CJ^\sharp\setminus\CJ$. Then $As\in\CJ$, hence $A\preceq_T As$ would imply $A\in\CJ$, as $\CJ$ is open, which contradicts our assumption. If $A\in\CJ\setminus\CJ^\flat$, then $As\not\in\CJ$, so $As\preceq_T A$ would contradict the fact that $\CJ$ is open. 

(2) The first identity $(\bigcup_{i\in I}\CJ_i)^\sharp=\bigcup_{i\in I}\CJ_i^\sharp$ is clear. It is also clear that  $(\bigcap_{i\in I}\CJ_i)^\sharp\subset \bigcap_{i\in I}\CJ_i^\sharp$.  So assume $A\in\bigcap_{i\in I}\CJ_i^\sharp$. This implies $As\in\bigcap_{i\in I} \CJ_i^\sharp$. As $(\bigcap_{i\in I}\CJ_i)^\sharp$ is $s$-invariant, it suffices to show that it contains either $A$ or $As$. So it suffices to show that it contains $A$ in the case $A\preceq_TAs$. But in this case $A\in\CJ_i^\sharp$ implies $A\in\CJ_i$ by (1), hence $A$ is even contained in $\bigcap_{i\in I}\CJ_i$. 

(3) is proven with similar arguments.  \end{proof}

\subsection{More admissible families}
Recall the definition of an admissible family in Section \ref{subsec-admfam}.

\begin{lemma} \label{lemma-rigsinv} The family of $s$-invariant $T$-open subsets in $\CA$ is $T$-admissible. \end{lemma}
\begin{proof} Properties (1), (2) and (3) of an admissible family are clear. It remains to check property (4). Let $\CJ\subset\CA_T$ be open and let $x$ be a $\DZ R$-orbit in $\CA$. We can replace $\CJ$ by the smallest open subset containing $\CJ\cap x$ and hence assume $\CJ=(\CJ\cap x)_{\preceq_T}$. Then $\CJ$ and $x$ are contained in the same connected component of $\CA_T$. We claim that  $(\CJ\cup\CJ s)\cap x=\CJ\cap x$. As $\CJ\cup\CJ s$ is open by Lemma \ref{lemma-Js} and clearly $s$-invariant, this serves our purpose.

If $x$ and $xs$ are not contained in the same component of $\CA_T$, then $\CJ s\cap x=\emptyset$ and the statement follows.  If $x$ and $xs$ are contained in the same component, then either $A\preceq_T As$ for all $A\in x$ or $As\preceq_T A$ for all $A\in x$ as the $\DZ R$-action preserves the partial order. In the second case, the set $\CJ$ is $s$-invariant since $\CJ=\bigcup_{A\in\CJ\cap x}\{\preceq_T A\}$ and since, by Lemma \ref{lemma-tops}, the set $\{\preceq_T A\}$ is $s$-invariant for all $A\in x$, so $(\CJ\cup\CJ s)\cap x=\CJ\cap x$. So  assume that $A\preceq_T As$ for all $A\in x$. Let $B\in\CJ s\cap x$. Then $B\preceq_T Bs$ and $Bs\in\CJ$, hence $B\in\CJ$ as $\CJ$ is open. So $\CJ s\cap x\subset\CJ\cap x$ and hence $(\CJ\cup\CJ s)\cap x=\CJ\cap x$. 
 \end{proof}

\begin{lemma} \label{lemma-admtwice} Let $T\to T^\prime$ be a homomorphism of base rings. Then the family of $s$-invariant $T$-open subsets of $\CA$ is $T^\prime$-admissible.  \end{lemma}
\begin{proof} Again, the properties (1), (2) and (3) of an admissible family are easily checked. Let $\CJ$ be a $T^\prime$-open set and $x\in\CV$. By Lemma \ref{lemma-Top} there is a $T$-open set $\CJ^\prime$ with $\CJ\cap x=\CJ^\prime\cap x$. By Lemma \ref{lemma-rigsinv} there is a $T$-open $s$-invariant subset $\CJ^{\prime\prime}$ with $\CJ^{\prime\prime}=\CJ^\prime\cap x$. This is what we wanted to show. 
\end{proof}
\subsection{$s$-invariants in $\CZ$} \label{subsec-sinvinZ} Recall that we have a right action of $s$ on the set $\CV$ of $\DZ R$-orbits in $\CA$. Suppose $\CL\subset\CV$ is $s$-invariant (i.e. $\CL=\CL s$).
Define $\eta_s\colon\bigoplus_{x\in\CL}T\to\bigoplus_{x\in\CL}T$  by $\eta_s(z_x)=(z^\prime_x)$, where $z^\prime_x=z_{xs}$ for all $x\in\CL$. 

\begin{lemma}  $\eta_s$ preserves the subalgebra $\CZ^\CL\subset\bigoplus_{x\in\CL}T$. 
\end{lemma}
\begin{proof} This follows directly from the definition in the case $T=S$ and $\CL=\CV$, as the right action of $s$ on $\CV$ commutes with the left action of $\CW$.  It  follows for arbitrary $T$ and $\CL=\CV$. Clearly, the projection $\bigoplus_{x\in\CV} T \to\bigoplus_{x\in\CL} T$ intertwines the action of $\eta_s$ on both spaces for an arbitrary $s$-invariant $\CL$. The claim follows. \end{proof} 

 \begin{definition}
We denote by $\CZ^{\CL,s}$ the sub-$T$-algebra of $\eta_s$-fixed elements, and by $\CZ^{\CL,-s}\subset\CZ^{\CL}$ the sub-$\CZ^{\CL,s}$-module of $\eta_s$-antiinvariants (i.e. the set of elements $z$ with $\eta_s(z)=-z$). We write $\CZ^s$ and $\CZ^{-s}$ in the case $\CL=\CV$.
\end{definition}
As $2$ is invertible in $k$ we have $\CZ^{\CL}=\CZ^{\CL,s}\oplus\CZ^{\CL,-s}$. 

\begin{lemma}\label{lemma-lstrucfree}  We have $\CZ^{\CL,-s}\cong\CZ^{\CL,s}[-2]$ as a $\CZ^{\CL,s}$-module. In particular, $\CZ^{\CL}\cong\CZ^{\CL,s}\oplus\CZ^{\CL,s}[-2]$ as a $\CZ^{\CL,s}$-module. \end{lemma}
Note that here the grading shift only makes sense, of course, if $T$ is a graded $S$-algebra. Otherwise, the statement holds without the shift. This remark applies to each occurrence of grading shifts in the following.
\begin{proof} As $\CZ_S$ is the  structure algebra associated with the root system $R$ the statement translates into \cite[Lemma 5.1 \& Proposition 5.3]{FieTAMS}.
\end{proof}

\begin{remark}\label{rem-trans} For a $\CZ$-module $M$ that is $\CZ$-supported on an $s$-invariant set $\CL$ it follows that $\CZ^{\CL}\otimes_{\CZ^{\CL,s}}M\cong M\oplus M[-2]$ as a $\CZ^{\CL,s}$-module, and hence as a $T$-module. So the functor $\CZ^{\CL}\otimes_{\CZ^{\CL,s}} (\cdot)$ is exact as an endofunctor of the category of  $\CZ^{\CL}$-modules and preserves root torsion freeness and root reflexivity.
\end{remark}

\begin{lemma}\label{lemma-transsinv} Let $M$ be a  root torsion free $\CZ$-module that is $\CZ$-supported on $\CL$. Then there is a canonical isomorphism $\CZ\otimes_{\CZ^s} M=\CZ^{\CL}\otimes_{\CZ^{\CL,s}}M$. In particular, $\CZ\otimes_{\CZ^s} M$ is $\CZ$-supported on $\CL$ as well. 
\end{lemma}
\begin{proof} Note that the $\CZ$-action on $M$ factors over the homomorphism $\CZ\to\CZ^{\CL}$, hence we obtain a homomorphism $\CZ\otimes_{\CZ^s}M\to \CZ^{\CL}\otimes_{\CZ^{\CL,s}}M$ that clearly is surjective.  Remark \ref{rem-trans} now implies that this  is an isomorphism. 
\end{proof}

\begin{lemma} \label{lemma-decZ} Suppose that $\Lambda\subset\CA_T$ is a union of connected components with the property $\Lambda\cap \Lambda s=\emptyset$. Then the composition $\CZ^{\ol{\Lambda\cup\Lambda s},s}\subset\CZ^{\ol{\Lambda\cup\Lambda s}}=\CZ^{\ol\Lambda}\oplus\CZ^{\ol{\Lambda s}}\xrightarrow{pr_{\ol\Lambda}}\CZ^{\ol\Lambda}$ is an isomorphism of $T$-algebras. 
\end{lemma}
\begin{proof} It follows from the definition that the isomorphism $\gamma\colon\bigoplus_{x\in{\ol\Lambda}}T\to\bigoplus_{y\in\ol{\Lambda s}} T$ that maps $(z_x)$ into $(z^\prime_y)$ with $z^\prime_y=z_{ys}$ induces an isomorphism $\CZ^{\ol\Lambda}\xrightarrow{\sim} \CZ^{\ol{\Lambda s}}$ of $T$-algebras. As $\CZ^{\ol{\Lambda\cup\Lambda s},s}\subset\CZ^{\ol\Lambda}\oplus\CZ^{\ol{\Lambda s}}$ is the subset of elements $(z,\gamma(z))$, the claim follows.
\end{proof}

\subsection{$s$-invariant elements in $\CZ\otimes_{\CZ^s}M$} \label{subsec-sinv}

Let $M$ be a $\CZ$-module. Then the map $\eta_s^M\colon \CZ\otimes_{\CZ^s} M\to \CZ\otimes_{\CZ^s} M$, $z\otimes m\mapsto \eta_s(z)\otimes m$, is a $\CZ^s$-linear automorphism of $\CZ\otimes_{\CZ^s} M$.

\begin{definition} 
An element $m\in\CZ\otimes_{\CZ^s}M$ is called {\em $s$-invariant} if $\eta_s^M(m)=m$.  
\end{definition} 



\begin{lemma} \label{lemma-sinvel} Let $\CL\subset\CV$ and let $M$ be a root torsion free $\CZ$-module. \begin{enumerate}
\item $\eta_s^M$ induces an isomorphism between $(\CZ\otimes_{\CZ^s} M)^{\CL}$ and $(\CZ\otimes_{\CZ^s} M)^{\CL s}$. 
\item If $\CL\cap\CL s=\emptyset$,  then an element $(m_\CL,m_{\CL s})\in (\CZ\otimes_{\CZ^s} M)^{\CL\cup\CL s}\subset (\CZ\otimes_{\CZ^s}M)^{\CL}\oplus(\CZ\otimes_{\CZ^s}M)^{\CL s}$ is $s$-invariant if and only if $m_{\CL s}=\eta_s^M(m_\CL)$.
\end{enumerate} 
\end{lemma}
\begin{proof} As $(\CZ\otimes_{\CZ^s}M)^{\CI}=\CZ^{\CI}\otimes_{\CZ^s} M$ the statement (1) follows from the fact that $\eta_s$ induces an  isomorphism between $\CZ^{\CL}$ and $\CZ^{\CL s}$ (of $\CZ^s$-algebras). Part (1)  implies part (2). \end{proof}

The following result is a rather technical statement that is used later on. 
\begin{lemma}\label{lemma-decm} Let $\CL=\{x,xs\}\subset\CV$ and  let $M$ be a root torsion free $\CZ$-module that is $\CZ$-supported on $\CL$.  For each $m\in\CZ\otimes_{\CZ^s} M$ there is a unique $s$-invariant element $m^\prime\in\CZ\otimes_{\CZ^s} M$ such that $m_x=m^\prime_x$.  
\end{lemma}
\begin{proof} As $M$ is $\CZ$-supported on $\CL$  we have $\CZ\otimes_{\CZ^s}M=\CZ^\CL\otimes_{\CZ^{\CL, s}} M$ by Lemma \ref{lemma-transsinv}. Let $\alpha\in R^+$ be such that $xs=s_{\alpha} x$. By Lemma \ref{lemma-subgensur} we have $$
\CZ^{\CL}=\{(z_x,z_{xs})\in T\oplus T\mid z_x\equiv z_{xs}\mod \alpha^\vee\}
$$
and hence $\CZ^{\CL,s}=T(1,1)$.
The elements $(1,1)$ and $(0,\alpha^\vee)$ form a $T$-basis of  $\CZ^{\CL}$. So we can write $m=(1,1)\otimes a+(0,\alpha^\vee)\otimes b$ with  $a,b\in M$. Clearly, $(1,1)\otimes a$ is $s$-invariant and $((0,\alpha^\vee)\otimes b)_{x}=0$. Hence $m_x=((1,1)\otimes a)_x$. This shows the existence. 
In order to show uniqueness it suffices to show that if $n\in (\CZ\otimes_{\CZ^s} M)^{\CL}$ is $s$-invariant and such that $n_x=0$, then $n=0$.  This is a consequence of Lemma \ref{lemma-sinvel}. 
\end{proof}

\subsection{Twisting (pre-)sheaves} Suppose that $\Lambda$ is a union of connected components of $\CA_T$ that satisfies $\Lambda\cap\Lambda s=\emptyset$. The map $\gamma_s\colon \Lambda\to\Lambda s$, $A\mapsto As$ is then a homeomorphism of topological spaces by Lemma \ref{lemma-tops}. Hence the pull-back functor $\gamma_s^{\ast}$ is an equivalence between the categories of (pre-)sheaves of $\CZ$-modules on $\Lambda s$ and (pre-)sheaves of $\CZ$-modules on $\Lambda$. For a $\CZ$-module $M$ denote by $M^{s-tw}$ the $\CZ$-module that we obtain from $M$ by twisting the $\CZ$-action with the automorphism $\eta_s\colon\CZ\to\CZ$ (cf. Section \ref{subsec-sinvinZ}), i.e. $M^{s-tw}=M$ as an abelian group and $z\cdot_{M^{s-tw}} m=\eta_s(z)\cdot_M m$ for $z\in\CZ$ and $m\in M$. 

Let $\SM$ be an object in $\bP$ supported on $\Lambda s$ and denote by $\gamma_s^{[\ast]}\SM$ the presheaf of $\CZ$-modules on $\Lambda\subset\CA_T$ that we obtain from the pull-back $\gamma_s^\ast\SM$ by applying the functor $(\cdot)^{s-tw}$ to any local section. Hence, for any open subset $\CJ$ of $\Lambda$ we have
$$
(\gamma_s^{[\ast]}\SM)(\CJ)=\SM(\CJ s)^{s-tw}.
$$

\begin{lemma}\label{lemma-gamma} Let $\SM$ be an object in $\bP$ supported  on $\Lambda s$. 
\begin{enumerate}
\item The presheaf $\gamma_s^{[\ast]}\SM$ is an object in $\bP$ supported   on $\Lambda$.
\item  If $\SM$ is a sheaf, then $\gamma_s^{[\ast]}\SM$ is a sheaf as well.
\end{enumerate}
Let $T\to T^\prime$ be a flat homomorphism of base rings. 
\begin{enumerate}\setcounter{enumi}{2}
\item There is a functorial isomorphism  $(\gamma_s^{[\ast]}\SM)\boxtimes_TT^\prime\cong \gamma_s^{[\ast]}(\SM\boxtimes_TT^\prime)$.
\item If $\SM$ is an object in $\bS$, then $\gamma_s^{[\ast]}\SM$ is an object in $\bS$ as well.
\end{enumerate}

\end{lemma}
\begin{proof} (1) Clearly $\gamma_s^{[\ast]}\SM$ is finitary, root torsion free,  flabby and supported  on $\Lambda$. We need to show that it satisfies the support condition. Let $x\in{\ol\Lambda}$. Note that $(\gamma_s^{[\ast]}\SM)^x=(\gamma_s^\ast\SM)^{xs}=\gamma_s^{\ast}(\SM^{xs})$.  As $i_{x}\circ\gamma_s=\gamma_s\circ i_{xs}$, we can identify the morphism $(\gamma_s^{[\ast]}\SM)^{x}(\CJ)\to i_{x\ast}i_x^\ast(\gamma_s^{[\ast]}\SM)^{x}(\CJ)$  with the morphism $\SM^{xs}(\CJ s)\to i_{xs\ast}i_{xs}^\ast\SM^{xs}(\CJ s)$ on the level of $\CZ^s$-modules for any open subset $\CJ$ of $\Lambda$. As $\SM$ satisfies the support condition, the latter is an isomorphism.

(2) This follows from the fact that $\gamma_s$, by Lemma \ref{lemma-tops},  is a homeomorphism of topological spaces. 

(3) Note that $\Lambda$ is also a union of connected components of $\CA_{T^\prime}$ by Lemma \ref{lemma-comp}. The presheaves $(\gamma_s^{[\ast]}\SM)\boxtimes_TT^\prime$ and $\gamma_s^{[\ast]}(\SM\boxtimes_TT^\prime)$ are both supported   on $\Lambda$. So let $\CJ\subset\Lambda$ be a $T$-open subset. Then $\CJ s\subset\Lambda s$ is $T$-open as well, and we have functorial identifications
\begin{align*}
((\gamma_s^{[\ast]}\SM)\boxtimes_TT^\prime)(\CJ)&=(\gamma_s^{[\ast]}\SM)(\CJ)\otimes_TT^\prime\\
&=\SM(\CJ s)^{s-tw}\otimes_TT^\prime\\
&=(\SM\otimes_TT^\prime)(\CJ s)^{s-tw}\\
&=(\SM\boxtimes_TT^\prime)(\CJ s)^{s-tw}\\
&=\gamma_s^{[\ast]}(\SM\boxtimes_TT^\prime)(\CJ)
\end{align*}
that are compatible with the restriction homomorphisms. 
As the set of $T$-open subsets in $\CA$ is $T^\prime$-admissible by Lemma \ref{lemma-Top}, the statement follows from Proposition \ref{prop-rig}.

(4) Clearly $\gamma_s^{[\ast]}\SM$ is root reflexive. Now the claim follows from this fact together with (2) and (3).
\end{proof}

\section{Wall crossing  functors}\label{sec-wcf}
In this  section we construct wall crossing functors on $\bP_T$ and show that they preserve the subcategory $\bS_T$ (they do not preserve the sheaf property in general). Wall crossing functors are of major importance in the companion article \cite{FieLanModRep}, where they are used to show that there are enough projective objects in $\bS_T$ and to link $\bS_T$ to the category defined by Andersen, Jantzen and Soergel.

\subsection{Characterization of the wall crossing functor}
  Fix a base ring $T$ and $s\in\hCS$.
Recall that for a $T$-open subset $\CJ$ the set $\CJ^\sharp=\CJ\cup\CJ s$ is $T$-open again by Lemma \ref{lemma-Js}. Let $\SM$ be an object in $\bP$. Denote by $\epsilon_s\SM$ the presheaf of $\CZ$-modules on $\CA_T$ given by 
$$
(\epsilon_s\SM)(\CJ):=\CZ\otimes_{\CZ^s}\SM(\CJ^\sharp)
$$
for $\CJ$ open in $\CA_T$, with restriction homomorphism  $\id_{\CZ}\otimes r^{\CJ^{\prime\sharp}}_{\CJ^\sharp}$  for an inclusion $\CJ^\prime\subset\CJ$. 
Then $\epsilon_s\SM$ is root torsion free  by  Remark \ref{rem-trans}  and it is flabby as the functor $\CZ\otimes_{\CZ^s}\cdot$ is right exact on the category of $\CZ$-modules. Clearly $\epsilon_s\SM$  is finitary, but in general it does not satisfy the support condition. Define $$\vartheta_s\SM:=(\epsilon_s\SM)^+.$$ This is now an object in $\bP$ and we obtain a functor $\vartheta_s\colon\bP\to\bP$ which is called the {\em wall crossing functor} associated with $s$. The natural transformation $\id\to(\cdot)^+$ induces a natural transformation $\rho_s\colon\epsilon_s\to\vartheta_s$. 

\begin{theorem}\label{thm-main} Let $s\in\hCS$. The functor $\vartheta_s\colon\bP\to\bP$ is the  up to isomorphism unique functor  that admits a natural transformation $\rho_s\colon \epsilon_s\to\vartheta_s$ such that for every object $\SM$ in $\bP$ and any $s$-invariant open subset $\CJ$, the homomorphism $\rho_s^\SM(\CJ)\colon \epsilon_s\SM(\CJ)\to\vartheta_s\SM(\CJ)$ is an isomorphism of $\CZ$-modules. 
\end{theorem}
\begin{proof} Note that the uniqueness statement follows from Proposition \ref{prop-rig} and the fact that the set of $s$-invariant open subsets in $\CA_T$ is $T$-admissible  by Lemma \ref{lemma-rigsinv}.  It remains to show the following: Let  $\SM$ be an object in $\bP$ and $\CJ$ an $s$-invariant open subset of $\CA_T$. Then  $\rho_s^\SM(\CJ)\colon\epsilon_s\SM(\CJ)\to\vartheta_s\SM(\CJ)$  is an isomorphism. Using Lemma \ref{lemma-inj} this translates into the statement that the natural map 
$$
\epsilon_s\SM(\CJ)\to\bigoplus_{x\in\CV}\epsilon_s\SM((\CJ\cap x)_{\preceq_T})^x$$ (direct sum of restrictions and surjections onto stalks) is injective. 

Let  $x\in\CV$. Now $\tCJ_x:=((\CJ\cap x)_{\preceq_T})^\sharp$ is the smallest $s$-invariant open subset that contains $\CJ\cap x$. Since  $\CJ$ is $s$-invariant,    $\tCJ_x=(\CJ\cap x)_{\preceq_T}\cup(\CJ\cap xs)_{\preceq_T}$.  Then $\tCJ_x=\tCJ_{xs}$. By definition, 
$
\epsilon_s\SM((\CJ\cap x)_{\preceq_T})^x=\left(\CZ\otimes_{\CZ^s}\SM(\tCJ_x)\right)^x.
$
Note that  $\pi(\CJ\setminus\tCJ_x)\subset\CV\setminus\{x,xs\}$ as $\tCJ_x$ contains $\CJ\cap x$ and $\CJ\cap xs$. By Lemma \ref{lemma-suppcond}, the kernel of the restriction $\SM(\CJ)\to\SM(\tCJ_x)$ is hence $\CZ$-supported on $\CV\setminus\{x,xs\}$. As the kernel of $\SM(\tCJ_x)\to\SM(\tCJ_x)^{x,xs}$ is supported on $\CV\setminus\{x,xs\}$ as well, so  is the kernel of the composition  $\SM(\CJ)\to\SM(\tCJ_x)\to\SM(\tCJ_x)^{x,xs}$. Taking the direct sum of all these homomorphisms yields hence an {\em injective} homomorphism
$$
\SM(\CJ)\to\bigoplus_{\{x,xs\}\in\CV/s}\SM(\tCJ_x)^{x,xs}.
$$
 As the functor $\CZ\otimes_{\CZ^s}(\cdot)$ is left exact on the category of $\CZ$-modules, by Remark \ref{rem-trans}, also the homomorphism
$$
\CZ\otimes_{\CZ^s}\SM(\CJ)\to\bigoplus_{\{x,xs\}\in\CV/s}\CZ\otimes_{\CZ^s}(\SM(\tCJ_x)^{x,xs})
$$
is injective. As $\{x,xs\}$ is an $s$-invariant subset of $\CV$, Lemma \ref{lemma-transsinv} implies that $\CZ\otimes_{\CZ^s}(\SM(\tCJ_x)^{x,xs})=(\CZ\otimes_{\CZ^s}\SM(\tCJ_x))^{x,xs}$. The canonical homomorphism $N^{x,xs}\to N^x\oplus N^{xs}$ is injective for any root torsion free $\CZ$-module $N$, so the composition 
$$
\CZ\otimes_{\CZ^s}\SM(\CJ)\to \bigoplus_{\{x,xs\}\in\CV/s}(\CZ\otimes_{\CZ^s}\SM(\tCJ_x))^{x,xs}\to \bigoplus_{x\in\CV}(\CZ\otimes_{\CZ^s}\SM(\tCJ_x))^{x}
$$
is injective. But this is the homomorphism $\epsilon_s\SM(\CJ)\to\bigoplus_{x\in\CV}\epsilon_s\SM((\CJ\cap x)_{\preceq_T})^x$.
\end{proof}

\subsection{Wall crossing and base change}
Suppose that $T\to T^\prime$ is a flat homomorphism of base rings. We have two wall crossing functors $\vartheta_{T,s}\colon \bP_T\to\bP_T$ and $\vartheta_{T^\prime,s}\colon\bP_{T^\prime}\to\bP_{T^\prime}$ and a base change functor $\cdot\boxtimes_TT^\prime\colon\bP_T\to\bP_{T^\prime}$. 
\begin{lemma} \label{lemma-wcbc} The wall crossing functors commute naturally with base change, i.e. $\vartheta_{T,s}(\cdot)\boxtimes_TT^\prime\cong\vartheta_{T^\prime,s}(\cdot\boxtimes_TT^\prime)\colon\bP_T\to\bP_{T^\prime}$. 
\end{lemma}
\begin{proof} Let $\SM$ be an object in $\bP_T$ and let $\CJ$ be an $s$-invariant $T$-open subset. Then
\begin{align*}
((\vartheta_{T,s}\SM)\boxtimes_TT^\prime)(\CJ)&=(\vartheta_{T,s}\SM)(\CJ)\otimes_TT^\prime\\
&=(\CZ_T\otimes_{\CZ_T^s}\SM(\CJ))\otimes_TT^\prime\\
&=\CZ_{T^\prime}\otimes_{\CZ_{T^\prime}^s}(\SM(\CJ)\otimes_TT^\prime)\\
&=\CZ_{T^\prime}\otimes_{\CZ_{T^\prime}^s}(\SM\boxtimes_TT^\prime)(\CJ)\\
&=\vartheta_{T^\prime,s}(\SM\boxtimes_TT^\prime)(\CJ).
\end{align*}
The set of $s$-invariant open subsets in $\CA_T$ is $T^{\prime}$-admissible by Lemma \ref{lemma-admtwice}. Hence the statement of the lemma follows from Proposition \ref{prop-rig}.
\end{proof}

\subsection{Wall crossing functors in the case $\Lambda\cap\Lambda s=\emptyset$}

Let $\Lambda$ be a union of connected components of $\CA_T$ and suppose $\Lambda\cap \Lambda s=\emptyset$.

\begin{lemma}\label{lemma-decgen}   Let $\SM$ be an object in $\bP$. Then 
$$
(\vartheta_s\SM)^{\ol\Lambda}\cong\SM^{\ol\Lambda}\oplus \gamma_s^{[\ast]}(\SM^{\ol{\Lambda s}}).
$$
\end{lemma}
\begin{proof} As $\vartheta_s\SM$, $\SM$ and $\gamma_s^{[\ast]}(\SM^{\ol{\Lambda s}})$ are objects in $\bP$, it is sufficient, by Proposition \ref{prop-rig}, to establish an isomorphism 
$$
(\vartheta_s\SM)^{\ol\Lambda}(\CJ)\cong\SM^{\ol\Lambda}(\CJ)\oplus\gamma_s^{[\ast]}(\SM^{\ol{\Lambda s}})(\CJ)
$$
for any $s$-invariant $T$-open subset $\CJ$ in a way that  is compatible with restrictions. 
So let $\CJ$ be $s$-invariant. We have functorial identifications
\begin{align*}
(\vartheta_s\SM)^{\ol\Lambda}(\CJ)&=(\vartheta_s\SM)(\CJ)^{\ol\Lambda}=(\epsilon_s\SM)(\CJ)^{\ol\Lambda}\\
&=(\CZ\otimes_{\CZ^s}\SM(\CJ))^{\ol\Lambda}\\
&=\CZ^{\ol\Lambda}\otimes_{\CZ^s}\SM(\CJ)\\
&=\CZ^{\ol\Lambda}\otimes_{\CZ^{\ol{\Lambda\cup\Lambda s},s}}(\SM^{\ol\Lambda}(\CJ)\oplus\SM^{\ol{\Lambda s}}(\CJ))\\
&=\SM^{\ol\Lambda}(\CJ)\oplus\SM^{\ol{\Lambda s}}(\CJ)^{s-tw}\\
&=\SM^{\ol\Lambda}(\CJ)\oplus\gamma_s^{[\ast]}(\SM^{\ol{\Lambda s}})(\CJ).
\end{align*}
In the sixth  equation we used Lemma \ref{lemma-decZ} (note that $\eta_s$ induces an isomorphism $\CZ^{\ol\Lambda}\xrightarrow{\sim}\CZ^{\ol{\Lambda s}}$ of $T$-algebras, so $\SM^{\ol{\Lambda s}}(\CJ)^{s-tw}$ is a $\CZ^{\ol\Lambda}$-module).
\end{proof}

\begin{lemma}\label{lemma-subgensheaf}
 If $\SM$ is an object in $\bS$ that is supported on a union $\Lambda$ of connected components that satisfies $\Lambda\cap\Lambda s=\emptyset$, then $\vartheta_s\SM$ is an object in $\bS$ again.
\end{lemma}

\begin{proof}  
Note that $\SN$ is an object in $\bS$ if and only if $\SN^{\ol\Lambda}$ and $\SN^{\ol{\Lambda s}}$ are objects in $\bS$. The statement then follows from Lemma \ref{lemma-decgen} and Lemma  \ref{lemma-gamma}. 
\end{proof} 

\subsection{Wall crossing functors and $s$-invariant open sets} 
It is not difficult to control the local sections of images of the wall crossing functors over $s$-invariant open sets. 
\begin{lemma}\label{lemma-wcrf} Let $\SM$ be an object in $\bP$. \begin{enumerate}
\item Suppose $\SM$ is root reflexive, and let  $\CJ$ be an $s$-invariant open subset of $\CA_T$. Then $\vartheta_s\SM(\CJ)$ is root reflexive.
\item Suppose that $\SM$ is a sheaf on $\CA_T$. Then $\vartheta_s\SM$ satisfies the sheaf property with respect to families of $s$-invariant open subsets in $\CA_T$.
\end{enumerate}
\end{lemma}
\begin{proof} Recall that $\CZ\cong\CZ^s\oplus \CZ^s[-2]$ as $\CZ^s$-modules by Lemma \ref{lemma-lstrucfree}. Fix such an isomorphism. 
Using Theorem \ref{thm-main},  for any $s$-invariant open subset $\CJ$  we  then have an isomorphism $\vartheta_s\SM(\CJ)=\SM(\CJ)\oplus \SM(\CJ)[-2]$ (of $\CZ^s$-modules). From this, part (1) follows. Note that the above isomorphism is  compatible with restrictions for inclusions $\CJ^\prime\subset\CJ$ of $s$-invariant open subsets. 
In order to prove claim (2) we need to show the following. Let $\CJ_i\subset\CA_T$ for $i\in I$ be $s$-invariant open subsets, and $m_i\in(\vartheta_s\SM){(\CJ_i)}$ sections that satisfy $m_i|_{\CJ_i\cap\CJ_j}=m_j|_{\CJ_i\cap\CJ_j}$ for all pairs $i,j\in I$. Then there exists a unique section $m\in (\vartheta_s\SM)(\bigcup_i \CJ_i)$ that satisfies $m|_{\CJ_i}=m_i$ for all $i\in I$.
By the above we have an identification $\vartheta_s\SM(\CJ_i)=\SM(\CJ_i) \oplus\SM(\CJ_i)[-2]$ of $\CZ^s$-modules that is compatible with the restriction homomorphisms. Hence we can write $m_i=m_i^\prime+m_i^{\prime\prime}$ with $m_i^\prime\in\SM{(\CJ_i)}$ and $m_i^{\prime\prime}\in\SM(\CJ_i)[-2]$ for all $i\in I$. As $\SM$ is a sheaf,  there are unique sections $m^\prime\in\SM({\bigcup_{i}\CJ_i})$ and $m^{\prime\prime}\in\SM(\bigcup_i\CJ_i)[-2]$ restricting to $m_i^\prime$ and $m_i^{\prime\prime}$ on $\CJ_i$, resp., for all $i\in I$. Hence $m:=m^\prime+m^{\prime\prime}$ is the unique extension of the $m_i$.
\end{proof}

\subsection{$s$-invariant sections}
If $\CJ$ is an $s$-invariant open subset, then the natural transformation $\rho_s\colon\epsilon_s\to\vartheta_s$ yields a natural identification $\vartheta_s\SM(\CJ)=\CZ\otimes_{\CZ^s}\SM(\CJ)$ by Theorem \ref{thm-main}, which is compatible with the restriction homomorphisms associated to an inclusion of $s$-invariant open subsets. So we can consider the endomorphism $\eta_s=\eta_s^{\CZ}\otimes\id$ on $\CZ\otimes_{\CZ^s}\SM({\CJ})$ constructed in Section \ref{subsec-sinv}. 

\begin{definition} For an $s$-invariant open subset $\CJ$ and an object  $\SM$ in $\bP$ we say that a section $m$ in $\vartheta_s\SM(\CJ)$ is {\em $s$-invariant} if $\eta_s(m)=m$, i.e. if $m$ is contained in $\CZ^s\otimes_{\CZ^s}\SM(\CJ)\subset\CZ\otimes_{\CZ^s}\SM(\CJ)$. 
\end{definition}

Let $\Lambda\subset \CA_T$ be a union of connected components. To simplify the notation, we write
$$
\CJ_\Lambda:=\CJ\cap\Lambda
$$
for an open subset $\CJ$ of $\CA_T$. 
Here is an easy characterization of $s$-invariant sections on $\CJ\subset \Lambda\cup\Lambda s$ in the case $\Lambda\cap\Lambda s=\emptyset$.  
 
\begin{lemma}\label{lemma-sinvgamma} Let $\Lambda$ be a union of components of $\CA_T$ with $\Lambda\cap\Lambda s=\emptyset$. For any open subset $\CJ$ of $\CA_T$ there is an isomorphism 
$$
\eta_s^{\ol\Lambda}\colon(\vartheta_s\SM)^{\ol\Lambda}(\CJ_\Lambda)\xrightarrow{\sim}(\vartheta_s\SM)^{\ol{\Lambda s}}(\CJ_\Lambda s)
$$ 
of $\CZ^s$-modules that is functorial in $\SM$ and compatible with restriction homomorphisms. Moreover, if $\CJ$ is $s$-invariant, then $\CJ_{\Lambda s}=\CJ_\Lambda s$ and an element $m=(m_{\ol\Lambda},m_{{\ol{\Lambda s}}})\in(\vartheta_s\SM)^{\ol\Lambda}(\CJ_\Lambda)\oplus(\vartheta_s\SM)^{\ol{\Lambda s}}(\CJ_{\Lambda s})$ is $s$-invariant if and only if $m_{{\ol{\Lambda s}}}=\eta^{\ol\Lambda}_s(m_{\ol\Lambda})$. 
\end{lemma}
\begin{proof} Suppose that $\CJ\subset\Lambda\cup\Lambda s$. Then $\CJ^{\sharp}=\CJ\dot\cup\CJ s$ and, by the support condition,  $\vartheta_s\SM^{\ol\Lambda}(\CJ)=\vartheta_s\SM^{\ol\Lambda}(\CJ^{\sharp})$.  As $\vartheta_s\SM(\CJ^{\sharp})=\CZ\otimes_{\CZ^s}\SM(\CJ^{\sharp})$ we can define $\eta_s^{\ol\Lambda}$ as in Lemma \ref{lemma-sinvel}. If $\CJ$ is $s$-invariant, then clearly $\CJ_{\Lambda s}=\CJ_\Lambda s$  and the remaining claim follows from Lemma \ref{lemma-sinvel} as well, after identifying $\vartheta_s\SM(\CJ)$ with $(\vartheta_s\SM)^{\ol{\Lambda}}(\CJ_\Lambda)\oplus(\vartheta_s\SM)^{\ol{\Lambda s}}(\CJ_{\Lambda s})$.   \end{proof}

The next result is the main technical result of this article and the most important ingredient in the proof that the wall crossing functors preserve the category $\bS$.

\begin{proposition} \label{prop-sinvpreim} Let $\SM$ be an object in $\bS$, let $\CJ\subset\CA_T$ be open and let $m\in\vartheta_s\SM(\CJ)$ be a section such that $m^\flat:=m|_{\CJ^\flat}$ is $s$-invariant. Then there exists a unique $s$-invariant section 
$m^\sharp\in\vartheta_s\SM(\CJ^\sharp)$ with $m^\sharp|_\CJ=m$.
\end{proposition}
\begin{proof}
Suppose we have proven the statement of the proposition in the cases that $T$ is either generic or subgeneric. For general $T$ we can then view $m$ as an element in $\vartheta_s\SM(\CJ)\otimes_TT^{\ast}=\vartheta_s(\SM\boxtimes_TT^{\ast})(\CJ)$ for any $\ast\in R^+\cup\{\emptyset\}$ and obtain unique preimages $m^{\sharp,\ast}$ in $\vartheta_s\SM(\CJ^\sharp)\otimes_TT^{\ast}=\vartheta_s(\SM\boxtimes_TT^\ast)(\CJ^\sharp)$ with $m^{\sharp,\ast}|_\CJ=m$. The uniqueness statement implies that we have actually found an element in $\bigcap_{\alpha\in R^+}\vartheta_s\SM(\CJ^\sharp)\otimes_TT^{\alpha}$. By Lemma \ref{lemma-wcrf}, $\vartheta_s\SM(\CJ^\sharp)$ is root reflexive, so this is an element $m^\sharp\in\vartheta_s\SM(\CJ^\sharp)$ with $m^\sharp|_\CJ=m$. The uniqueness of $m^\sharp$ is implied by the uniqueness statement for $\ast=\emptyset$.

So we can now assume that $T$ is either generic or subgeneric.  Consider the canonical decomposition $\vartheta_s\SM=\bigoplus_{\Lambda\in C(\CA_{T})}(\vartheta_s\SM)^{\ol\Lambda}$. We can assume that $m$ is supported on $\Lambda\cup\Lambda s$ for some $\Lambda\in C(\CA_T)$. Moreover, by the support condition we can also assume that $\CJ$ is a subset of $\Lambda\cup\Lambda s$. We now distinguish the cases $\Lambda=\Lambda s$ and $\Lambda\ne\Lambda s$. 

First assume that $\Lambda=\Lambda s$. 
Then $T$ must be subgeneric and $\ol\Lambda=\{x,xs\}$  for some $\DZ R$-orbit $x$. 
  Let $\tilde m^\sharp\in\vartheta_s\SM(\CJ^\sharp)=\CZ\otimes_{\CZ^s}\SM(\CJ^\sharp)$ be an arbitrary preimage of $m$. Lemma \ref{lemma-decm} allows us to write $\tilde m^\sharp=m_1+m_2$  with an $s$-invariant element $m_1$ and an element $m_2$ that is $\CZ$-supported on $\{xs\}$. Now $m|_{\CJ^\flat}=m_1|_{\CJ^\flat}+m_2|_{\CJ^\flat}\in\vartheta_s\SM(\CJ^\flat)=\CZ\otimes_{\CZ^s}\SM(\CJ^\flat)$ is $s$-invariant,  and the uniqueness statement in Lemma \ref{lemma-decm} implies that $m_2|_{\CJ^\flat}=0$. In particular, $m_2|_{\CJ}$ is contained in the kernel of $\vartheta_s\SM(\CJ)\to\vartheta_s\SM(\CJ^\flat)$. As $\vartheta_s\SM$ satisfies the support condition, Lemma \ref{lemma-suppcond}  implies that $\supp_{\CZ}\,m_2|_{\CJ}\subset\{x\}$. But by construction $\supp_{\CZ}\,m_2|_{\CJ}\subset\{xs\}$. This is only possible if  $m_2|_\CJ=0$. It follows that $m_1|_\CJ=\tilde m^\sharp|_\CJ=m$, and so $m^\sharp:=m_1$ is  an $s$-invariant preimage of $m$.  The uniqueness of $m^\sharp$ again follows from Lemma \ref{lemma-decm}.

Now suppose that $\Lambda\ne\Lambda s$. Then we can write $m=m_{\ol\Lambda}+m_{{\ol{\Lambda s}}}$ with $m_{\ol\Lambda}\in(\vartheta_s\SM)^{\ol\Lambda}(\CJ)$ and $m_{{\ol{\Lambda s}}}\in(\vartheta_s\SM)^{\ol{\Lambda s}}(\CJ)$. We have $(\vartheta_s\SM)^{\ol\Lambda}(\CJ)\cong(\vartheta_s\SM)^{\ol\Lambda}(\CJ_{\Lambda})$ and $(\vartheta_s\SM)^{\ol{\Lambda s}}(\CJ)\cong(\vartheta_s\SM)^{\ol{\Lambda s}}(\CJ_{\Lambda s})$. Lemma \ref{lemma-sinvgamma} gives us isomorphisms $(\vartheta_s\SM)^{\ol\Lambda}(\CJ_\Lambda)\xrightarrow{\eta_s^{\ol\Lambda}}(\vartheta_s\SM)^{\ol{\Lambda s}}(\CJ_{\Lambda} s)$ and $(\vartheta_s\SM)^{\ol{\Lambda s}}(\CJ_{\Lambda s})\xrightarrow{\eta_s^{\ol{\Lambda s}}}(\vartheta_s\SM)^{\ol\Lambda}(\CJ_{\Lambda s}s)$. Now   $m_{\ol\Lambda}$ and $\eta_s^{\ol{\Lambda s}}(m_{{\ol{\Lambda s}}})$ are sections of $(\vartheta_s\SM)^{\ol\Lambda}$ over $\CJ_\Lambda$ and $\CJ_{\Lambda s}s$, resp. Note that $\CJ_\Lambda\cap \CJ_{\Lambda s}s=(\CJ^\flat)_\Lambda$ and $\CJ_{\Lambda}\cup\CJ_{\Lambda s}s=(\CJ^{\sharp})_\Lambda$. As $m_{\ol\Lambda}$ and $\eta_s^{\ol{\Lambda s}}(m_{{\ol{\Lambda s}}})$ agree on $\CJ^\flat$ (as $m|_{\CJ^\flat}$ is $s$-invariant), and as $(\vartheta_s\SM)^{\ol\Lambda}$ is a sheaf by Lemma \ref{lemma-subgensheaf}  the two sections glue and yield a section $m_{\ol\Lambda}^\sharp\in(\vartheta_s\SM)^{\ol\Lambda}((\CJ^\sharp)_\Lambda)=(\vartheta_s\SM)^{\ol\Lambda}(\CJ^\sharp)$. Similarly, with the roles of $\Lambda$ and $\Lambda s$ interchanged, we construct a section $m_{{\ol{\Lambda s}}}^\sharp$ in $(\vartheta_s\SM)^{\ol{\Lambda s}}(\CJ^\sharp)$ that restricts to $m_{{\ol{\Lambda s}}}$ and $\eta_s^{\ol\Lambda}(m_{\ol\Lambda})$ on $\CJ_{\Lambda s}$ and $\CJ_\Lambda s$, resp.  

Now 
\begin{align*}
\eta_s^{\ol\Lambda}(m_{\ol\Lambda}^\sharp)|_{\CJ_{\Lambda} s}=\eta_s^{\ol\Lambda}(m_{\ol\Lambda}^\sharp|_{\CJ_\Lambda})=\eta_s^{\ol\Lambda}(m_{\ol\Lambda})=m^\sharp_{\Lambda s}|_{\CJ_\Lambda s}.
\end{align*}
Analogously we obtain $\eta_s^{\ol{\Lambda s}}(m_{{\ol{\Lambda s}}}^\sharp)|_{\CJ_{\Lambda s}s}=m_{\ol\Lambda}^\sharp|_{\CJ_{\Lambda s}s}$, from which we deduce $\eta_s^{\ol\Lambda}(m_{\ol\Lambda}^\sharp)|_{\CJ_{\Lambda s}}=m^{\sharp}_{\Lambda s}|_{\CJ_{\Lambda s}}$. As $\vartheta_s\SM$ is a sheaf and since $\CJ_\Lambda s\cup\CJ_{\Lambda s}=(\CJ^\sharp)_{\Lambda s}$, we deduce $\eta_s^{\ol\Lambda}(m_{\ol\Lambda}^\sharp)=m_{{\ol{\Lambda s}}}^\sharp$. From Lemma \ref{lemma-sinvgamma} we deduce that $m^\sharp:=(m_{\ol\Lambda}^\sharp,m_{{\ol{\Lambda s}}}^\sharp)$ is an $s$-invariant section in $\vartheta_s\SM(\CJ^\sharp)$. By construction, $m^\sharp|_\CJ=m_{\ol\Lambda}+m_{{\ol{\Lambda s}}}=m$. 

In order to show uniqueness also in the case $\Lambda\ne\Lambda s$  it suffices to show that if $m^\sharp\in\vartheta_s\SM(\CJ^\sharp)$ is $s$-invariant and $m^\sharp|_\CJ=0$, then $m^\sharp=0$. But $m^\sharp|_{\CJ_\Lambda}=0$ implies $m^{\sharp}|_{\CJ_\Lambda\cup\CJ_{\Lambda} s}=0$ by $s$-invariance. Analogously $m^{\sharp}|_{\CJ_{\Lambda s}\cup\CJ_{\Lambda s} s}=0$. As $\CJ^\sharp=(\CJ_\Lambda\cup\CJ_{\Lambda} s)\cup(\CJ_{\Lambda s}\cup\CJ_{\Lambda s} s)$ and as both sets in parentheses are $s$-invariant, Lemma \ref{lemma-wcrf} implies that $m^\sharp=0$. \end{proof} 

\subsection{Wall crossing preserves the category $\bS$}
Now we state and prove one of the main results in this article. 
\begin{theorem} Suppose that $\SM$ is an object in $\bS$. Then $\vartheta_s\SM$ is an object in $\bS$ as well.
\end{theorem}
\begin{proof}  We can assume that $\SM$ is supported on the connected component $\Lambda$. The case $\Lambda\ne\Lambda s$ is already proven in Lemma  \ref{lemma-subgensheaf}. So now suppose that $\Lambda=\Lambda s$. Then $\vartheta_s\SM$ is supported on $\Lambda$ as well. First we show that  $\vartheta_s\SM$ is a  sheaf. 
Let  $\{\CJ_i\}_{i\in I}$ be a family of open subsets in $\Lambda$ and set $\CJ=\bigcup_{i\in I}\CJ_i$. Let  $m_i\in\vartheta_s\SM(\CJ_i)$ be sections with $m_i|_{\CJ_i\cap\CJ_j}=m_j|_{\CJ_i\cap\CJ_j}$ for all $i,j\in I$. Using Lemma \ref{lemma-wcrf} we see that the sections $m_i|_{\CJ_i^\flat}$ glue and yield a section $m^\flat\in\vartheta_s\SM(\CJ^\flat)$ (note that $\CJ^\flat=\bigcup_{i\in I}\CJ_i^\flat$ by Lemma \ref{lemma-sinvset}). As $\vartheta_s\SM$ is flabby, there is a preimage $m^\prime\in\vartheta_s\SM(\CJ)$ of $m^\flat$. Subtracting $m^\prime|_{\CJ_i}$ from $m_i$ shows that we can from now on assume that $m_i|_{\CJ_i^\flat}=0$ for all $i\in I$. In particular, each $m_i|_{\CJ_i^\flat}$ is $s$-invariant. For any $i\in I$ let $m_i^\sharp\in\vartheta_s\SM(\CJ_i^\sharp)$ be the unique $s$-invariant preimage of $m_i$ (cf. Proposition \ref{prop-sinvpreim}). We claim that  $m_i^\sharp|_{\CJ_i^\sharp\cap\CJ_j^\sharp}=m_j^\sharp|_{\CJ_i^\sharp\cap\CJ_j^\sharp}$ for all $i,j\in I$. Note that $\CJ_i^\sharp\cap\CJ_j^\sharp=(\CJ_i\cap\CJ_j)^\sharp$ by Lemma \ref{lemma-sinvset}, and $m_i^\sharp|_{\CJ_i\cap\CJ_j}=m_i|_{\CJ_i\cap\CJ_j}=m_j|_{\CJ_i\cap\CJ_j}=m_j^\sharp|_{\CJ_i\cap\CJ_j}$. From the uniqueness statement in Proposition \ref{prop-sinvpreim} we can now deduce  $m_i^\sharp|_{(\CJ_i\cap\CJ_j)^\sharp}=m_j^\sharp|_{(\CJ_i\cap\CJ_j)^\sharp}$. Using Lemma \ref{lemma-wcrf} again shows that the $m_i^\sharp$ glue and yield a section $m^\sharp\in\vartheta_s\SM(\CJ^\sharp)$. Then $m:=m^\sharp|_{\CJ}$ is the section we are looking for. 

Now we show that  $\vartheta_s\SM$ is root reflexive. 
We need to show that $\vartheta_s\SM(\CJ)$ is root reflexive for any open subset $\CJ$. Let $m$ be an element in $\bigcap_{\alpha\in R^+}\vartheta_s\SM(\CJ)\otimes_TT^{\alpha}$.  Then $m|_{\CJ^\flat}\in\vartheta_s\SM(\CJ^\flat)$ by Lemma \ref{lemma-wcrf}. By subtracting from $m$ a preimage of $m|_{\CJ^\flat}$ in $\vartheta_s\SM(\CJ)$ we can assume that $m|_{\CJ^\flat}=0$. Proposition  \ref{prop-sinvpreim} now shows that $m$ has a unique $s$-invariant preimage $m^{\ast\sharp}$ in $\vartheta_s\SM(\CJ^\sharp)\otimes_TT^{\ast}$ for all $\ast\in R^+\cup\{\emptyset\}$. By uniqueness, the elements $m^{\ast\sharp}$ agree as elements in $\vartheta_s\SM(\CJ^\sharp)\otimes_TT^{\emptyset}$, hence define an element in $\bigcap_{\alpha\in R^+}\vartheta_s\SM(\CJ^\sharp)\otimes_TT^{\alpha}$. Again by Lemma \ref{lemma-wcrf},   $\bigcap_{\alpha\in R^+}\vartheta_s\SM(\CJ^\sharp)\otimes_TT^{\alpha}=\vartheta_s\SM(\CJ^\sharp)$. So $m$ has a preimage $m^\sharp$ in $\vartheta_s\SM(\CJ^\sharp)$, so $m$ must be contained in $\vartheta_s\SM(\CJ)$. Hence $\vartheta_s\SM(\CJ)$ is root reflexive. 

To finish the proof, we need to show that $(\vartheta_s\SM)\boxtimes_TT^\prime$ is a root reflexive sheaf for any flat homomorphism $T\to T^\prime$ of base rings. As $(\vartheta_s\SM)\boxtimes_TT^\prime=\vartheta_{T^\prime,s}(\SM\boxtimes_TT^\prime)$ this follows from what we have already proven and the fact that $\SM\boxtimes_TT^\prime$ is a root reflexive sheaf (on $\CA_{T^\prime}$). 
\end{proof}

\end{document}